\documentclass[a4paper,12pt]{article}
\usepackage{amsmath,amssymb}
\usepackage{parskip,url}
\usepackage{cancel }
\usepackage{comment}
\usepackage[pdfview=Fit]{hyperref}
\hypersetup{bookmarkstype=none}
\newtheorem{theorem}{Theorem}[section]
\newtheorem{lemma}[theorem]{Lemma}
\newtheorem{proposition}[theorem]{Proposition}

\newcommand{\gr}{\operatorname{gr}}

\newcommand{\Sym}{\operatorname{Sym}}

\newcommand{\FM}{\operatorname{FM}}
\newcommand{\card}{\operatorname{card}}

\newenvironment{proof}{\par\noindent{\bf Proof.}}{$\qed$\par\bigskip}
\newcommand{\qed}{\enspace\vrule  height6pt  width4pt  depth2pt}
\begin{document}

\title{Finitely presented algebras defined by permutation relations of dihedral type}

\author{Ferran Ced\'o\footnote{Research partially supported by a grant of MICIIN (Spain)
MTM2011-28992-C02-01.} \and Eric Jespers\footnote{Research supported
in part by Onderzoeksraad of Vrije Universiteit Brussel and Fonds
voor Wetenschappelijk Onderzoek (Belgium).} \and Georg Klein}
\date{}
\maketitle

\begin{abstract}
The class of finitely presented algebras over a
field $K$ with a set of generators $a_{1},\ldots ,
a_{n}$ and defined by homogeneous relations of the
form
 $a_{1}a_{2}\cdots a_{n} =a_{\sigma (1)} a_{\sigma (2)} \cdots
 a_{\sigma (n)}$,
where $\sigma$ runs through a subset $H$ of the symmetric group
$\Sym_{n}$ of degree $n$, is investigated. Groups $H$ in which the
cyclic group $\langle (1,2, \ldots ,n) \rangle$ is a normal subgroup
of index $2$ are considered. Certain representations by permutations
of the dihedral and semidihedral groups belong to this class of
groups. A normal form for the elements of the underlying monoid
$S_n(H)$ with the same presentation as the algebra is obtained.
Properties of the algebra are derived, it follows that it is  an
automaton algebra in the sense of Ufnarovski\u{\i}. The universal
group $G_n$ of $S_n(H)$ is a unique product group, and it is the
central localization of a cancellative subsemigroup of
$S_n(H)$. This, together with previously obtained results on such
semigroups and algebras, is used to show that the algebra
$K[S_n(H)]$ is semiprimitive.
\end{abstract}


\section{Introduction}\label{intro}

We consider subgroups $H$ of $Sym_n$, the symmetric group of
order $n$, in which the cyclic subgroup $\langle (1, 2, \ldots ,n
)\rangle $ is normal and of index $2$. This class of groups includes
certain representations by permutations of the dihedral and
semi-dihedral groups. We study the properties of the monoid
with the following presentation:
$$S_n(H)= \langle a_1,a_2,\dots ,a_n\mid a_1a_2\cdots
a_n= a_{\sigma (1)}a_{\sigma (2)}\cdots a_{\sigma (n)},\; \sigma\in
H\rangle.$$
In Section~\ref{monoidd}, we establish a normal form for the elements of this semigroup.
The proof is an application of the diamond lemma and requires many combinatorial verifications.
Using arguments by analogy, the number of actual cases to verify were kept to a minimum.

In Section~\ref{alg1},  for a field $K$, we consider the algebra $K[S_n(H)]$ with the same presentation
as the semigroup. From the normal form it follows that
$K[S_n(H)]$ is an automaton algebra in the sense of Ufnarovski\u{\i} \cite{ufnar}.
The universal group $G_n$ of $S_n(H)$, by which we mean the group with the same presentation, is
a unique product group. The element $z=a_1a_2 \cdots a_n$ is central in $S_n(H)$, and
$G_n=zS_n(H) \langle z  \rangle ^{-1}$, with $zS_n(H)$ a cancellative subsemigroup of $S_n(H)$.
This is used to prove that $\mathcal{J}(K[S_n(H)])=0$,
also using the fact that  $\mathcal{J}(K[S_n(H)]) \subseteq K[zS_n(H)]$ from the reference literature.
We conclude Section~\ref{alg1} with some results regarding
prime ideals in the semigroup and in the algebra, which are related to further structural properties of the algebra.

Previous investigations of algebras of this type include
the cases of $H$ the full symmetric group \cite{alghomshort}, the alternating
group of degree $n$ \cite{alt4algebra,altalgebra}, and an abelian
group \cite{symcyclic}. In many cases the algebra turns out to be semiprimitive, with as noteworthy exception
the alternating group, where the Jacobson radical being zero or not depends on the parity of $n$ and the characteristic of the field $K$.
The subgroups $H$ considered in the present paper are of dihedral and quasi-dihedral type. In case the order of such groups
is $2^m$, there is only one other isomorphism class of groups which have the same order and which are also of nilpotency class $m-1$,
the generalized quaternion group $Q$. 
In \cite{quater}, we investigate $S_n(Q)$ and show that in this case more results can be obtained. Using different methods, it is shown that 
the semigroup $S_n(Q)$ is a two unique product semigroup if $Q$ is given via a regular representation by permutations.

\section{Monoids $S_n(H)$ defined by dihedral type relations}\label{monoidd}

In this section we study $S_{n}(H)$ in case $H$ contains the cyclic
group $\langle (1,2, \ldots , n) \rangle$ as a normal subgroup of
index $2$. Note that $n\geq 3$ and if $n=3$, then $H=\Sym_3$
and $S_3(H)=S_3(\Sym_3)$ which is studied in \cite{alghomshort}.
Hence, throughout we assume that $n>3$. So $H$ is generated by $\lambda =
(1,2, \ldots ,n)$ and $\mu $, with $\mu \notin \langle \lambda
\rangle$, and $\mu^2 \in \langle \lambda \rangle $. Since
$\langle \lambda \rangle $ is a normal subgroup of $H$, we have $\mu
\lambda = \lambda^k \mu$, for some integer $k$ such that $1\leq
k<n$ and $k^2 \equiv 1 \mod n$ (so $(k,n)=1$).
For ease of notation, arithmetic will be done modulo $n$ in
the set of indices $\{1,2, \ldots ,n \}$. We have
\begin{eqnarray*}
&&\mu(i+1)=\mu\lambda(i)=\lambda^k\mu(i)=\mu(i)+k,
\end{eqnarray*}
for all $i\in\{1,2, \ldots ,n \}$. So $k\neq 1$ as $\mu\notin \langle
\lambda\rangle$.
Thus, we obtain the following presentation
 \begin{equation*}
\begin{array}{r@{}l}
S_{n}(H)=  \langle a_1 ,a_2 ,\ldots ,a_n  \mid    a_1 a_2 \cdots a_{n-1}a_n
 &{} = a_i a_{i+1} a_{i+2} \cdots a_{i-2} a_{i-1} \vspace{6 pt} \\
 &{} =a_i a_{i+k} a_{i+2 k} \cdots  a_{i-2 k} a_{i- k} , \ 1 \leq i \leq n \rangle.
\end{array}
\end{equation*}
For simplicity, we denote $S_n(H)$ by $S_n$. It is clear that $S_n$ has only trivial units. The aim of this section is to obtain a normal
form for the elements of $S_{n}$. This requires tedious work. We begin by noting that
$$z=a_1a_2\cdots a_n$$ is a central element of $S_n$. Indeed, for all
$1\leq i\leq n$,
$$(a_1a_2\cdots a_n)a_i=a_ia_{i+1}\cdots
a_{i-2}a_{i-1}a_i=a_i(a_1a_2\cdots a_n).\label{centra}$$ For every integer
$p$, we define the following elements of length $n-2$ of $S_n$:
\begin{equation*}
\begin{array}{r@{}l}
b_{p}=&{} \ a_{p+1}a_{p+2}\cdots a_{p-3}a_{p-2}, \vspace{3pt} \\
c_{p}=&{} \ a_{p+k}a_{p+2k}\cdots a_{p-3k}a_{p-2k}.
\end{array}
\end{equation*}

\begin{lemma}\label{z2} For every non-negative integer $q$,
$$zb_{i}b_{i-(k+1)}\cdots b_{i-q(k+1)}a_{i-1-q(k+1)}=za_{i+k}c_{i+k}c_{i+k-(k+1)}\cdots c_{i+k-q(k+1)}$$ and
$$zc_{i}c_{i-(k+1)}\cdots c_{i-q(k+1)}a_{i-k-q(k+1)}=za_{i+1}b_{i+1}b_{i+1-(k+1)}\cdots b_{i+1-q(k+1)},$$
for $1 \leq i \leq n$.
\end{lemma}

\begin{proof}
We shall prove this by induction on $q$. For $q=0$, the statement is proved as follows.
\begin{equation*}
\begin{array}{r@{}l}
zb_{i}a_{i-1}&{}= za_{i+1}a_{i+2}\cdots a_{i-2}a_{i-1} \vspace{3 pt}   \\
&{}=(a_{i+k} a_{i+2k} \cdots a_{i-2k}a_{i-k}a_i)a_{i+1}a_{i+2}\cdots a_{i-2} a_{i-1} \vspace{3 pt} \\
&{}=a_{i+k} a_{i+2k} \cdots a_{i-2k}a_{i-k}z \vspace{3 pt} \\
&{}=za_{i+k}a_{i+2k}\cdots a_{i-2k} a_{i-k} \vspace{3 pt} \\
&{}=za_{i+k}c_{i+k}.
\end{array}
\end{equation*}
Similarly, one obtains that $zc_{i}a_{i-k}=za_{i+1}b_{i+1}$. Suppose
that $q\geq 1$ and the result is true for $q-1$. Since $z$ is
central, we have that
\begin{equation*}
\begin{array}{r@{}l}
&{}\hspace{-18 pt}  zb_{i}b_{i-(k+1)}\cdots b_{i-q(k+1)}a_{i-1-q(k+1)}  \vspace{3 pt} \\
&{}= b_{i}b_{i-(k+1)}\cdots b_{i-(q-1)(k+1)}zb_{i-q(k+1)}a_{i-1-q(k+1)} \vspace{3 pt}  \\
&{}= b_{i}b_{i-(k+1)}\cdots b_{i-(q-1)(k+1)}za_{i+k-q(k+1)}c_{i+k-q(k+1)} \vspace{3 pt} \\
&{}= zb_{i}b_{i-(k+1)}\cdots b_{i-(q-1)(k+1)}a_{i-1-(q-1)(k+1)}c_{i+k-q(k+1)} \vspace{3 pt} \\
&{}= za_{i+k}c_{i+k}c_{i+k-(k+1)}\cdots c_{i+k-(q-1)(k+1)}c_{i+k-q(k+1)}.
\end{array}
\end{equation*}
The other equality is proved similarly.
\end{proof}

Let $\FM_n$ denote the free monoid  on $\{ a_1,a_2,\dots ,a_n\}$ (we use
the same notation for the generators of $\FM_{n}$ and those of
$S_{n}$). Denote by $\ll$ the length-lexicographical order on $\FM_n$
generated by $$a_1\ll a_2\ll\dots\ll a_n.$$ Consider the following transformations of words in the
free monoid $\FM_n$ on $\{ a_1,a_2,\dots ,a_n\}$, that are dependent on the parameters $i$ (or $j$), $v$ and $q$,
with $1\leq i \leq n$, $v\in \FM_n$ and $q\geq 0$. Each transformation transforms a word into another
word which is smaller in the length-lexicographical order, and this is the reason why some restrictions
on $i$ (or $j$) are sometimes included.
The transformations are:
\begin{eqnarray}\label{ti}
t_i(a_{i} a_{i+1} \cdots a_{i-2} a_{i-1})=a_1a_2\cdots
a_n,
\end{eqnarray}
for $i=2,3, \ldots, n$;
\begin{eqnarray}\label{rjm}
r_{j,m}(a_{j}a_1^{m}a_2\cdots a_n)=a_1a_2\cdots a_na_ja_1^{m-1},
\end{eqnarray}
for $j=2,3,\dots ,n$ and $m\geq 1$;
\begin{eqnarray}\label{hi}
h_i(a_i a_{i+k} a_{i+2k} \cdots  a_{i-2 k} a_{i- k})=a_1a_2\cdots
a_n;
\end{eqnarray}
\begin{equation}\label{di}
\begin{array}{r@{}l}
&{} \hspace{-18 pt} d_{i,v,q}\left(a_1 \cdots a_n v  b_{i}b_{i-(k+1)} \cdots b_{i-q(k+1)} a_{i-1-q(k+1)}\right) \vspace{3 pt} \\
&{} =a_1 \cdots a_n v  a_{i+k}c_{i+k}c_{i+k-(k+1)} \cdots c_{i+k-q(k+1)},
\end{array}
\end{equation}
for $n-k+1 \leq i \leq n-1 $;
\begin{equation}\label{ei}
\begin{array}{r@{}l}
&{} \hspace{-18 pt}
e_{i,v,q}\left(a_1 \cdots a_n v  c_{i}c_{i-(k+1)} \cdots
c_{i-q(k+1)} a_{i-k-q(k+1)}\right) \vspace{3 pt}   \\
&{} =a_1 \cdots a_n v  a_{i+1}b_{i+1}b_{i+1-(k+1)}\cdots b_{i+1-q(k+1)},
\end{array}
\end{equation}
for $  0 \leq i \leq n-k $;
\begin{equation}\label{si}
\begin{array}{r@{}l}
&{} \hspace{-18 pt}
 s_{i,v,q}\left(a_1 \cdots a_n v a_i b_{i}b_{i-(k+1)} \cdots b_{i-q(k+1)} a_{i-1-q(k+1)}\right) \vspace{3 pt}   \\
&{}  =a_1^2 (a_2 \cdots a_n)^2 v   c_{i+k-(k+1)}c_{i+k-2(k+1)} \cdots c_{i+k-q(k+1)};
\end{array}
\end{equation}
\begin{equation}\label{ui}
\begin{array}{r@{}l}
&{}  \hspace{-18 pt}
 u_{i,v,q}\left(a_1 \cdots a_n v a_i c_{i}c_{i-(k+1)}\cdots
c_{i-q(k+1)} a_{i-k-q(k+1)}\right) \vspace{3 pt}   \\
&{}  =a_1^2(a_2 \cdots a_n)^2 v  b_{i+1-(k+1)}b_{i+1-2(k+1)} \cdots b_{i+1-q(k+1)}.
\end{array}
\end{equation}

It follows at once from the presentation\index{presentation} of $S_n$, the centrality\index{centrality} of $z=a_1 \cdots a_n$, and Lemma~\ref{z2}, that each of these
transformations maps a word in $\FM_{n}$ to another
word that represents the same element in $S_{n}$.

Let $w_1,w_2\in \FM_n$. We say that $w_1$ covers\index{covers} $w_2$ if $w_2$ is
obtained from $w_1$ by applying exactly one of the transformations\index{transformation}
$t_i,r_{j,m},h_i,d_{i,v,q},e_{i,v,q}, s_{i,v,q}, u_{i,v,q}$
to a subword of $w_1$. In this case we write $w_1\succ \label{coverr} w_2$. We
define a partial order in $\FM_n$ by $w\geq w'$ if and only if there
exist $w_0,w_1,\dots ,w_r\in \FM_n$, ($r\geq 0$), such that
$$w=w_0\succ w_1\succ\ldots\succ w_r=w'.$$
Note that for $w,w'\in
\FM_n$, $$w\geq w'\Rightarrow w\gg w'.$$

\begin{lemma}\label{joinn}
For an integer $q \geq 0$ and any word $v \in \FM_n $, let
\begin{equation*}
\left\{
\begin{array}{r@{}l}
w_1&{}= a_1 a_2 \cdots a_n v c_{i} c_{i-(k+1)} \cdots c_{i-q(k+1)} a_{i-k-q(k+1)} \in \FM_n \vspace{3 pt} \\
w_2&{}= a_1 a_2 \cdots a_n v a_{i+1}b_{i+1}b_{i+1-(k+1)} \cdots b_{i+1-q(k+1)} \in \FM_n \vspace{3 pt},
\end{array}
\right.
\end{equation*}
or
\begin{equation*}
\left\{
\begin{array}{r@{}l}
w_1&{}= a_1 a_2 \cdots a_n v b_{i} b_{i-(k+1)} \cdots b_{i-q(k+1)} a_{i-1-q(k+1)} \in \FM_n  \vspace{3 pt} \\
w_2&{}= a_1 a_2 \cdots a_n v a_{i+k}c_{i+k} c_{i+k-(k+1)} \cdots
c_{i+k-q(k+1)} \in \FM_n.
\end{array}
\right.
\end{equation*}
Then there exists $w_3 \in \FM_n$, such that $w_1, w_2 \geq w_3$, and $w_3$ begins with $a_1a_2 \cdots a_n$.
\end{lemma}
\begin{proof}
We will prove the first case, the second case can be proved similarly. Recall that
$$c_{i-p(k+1)}=a_{i+k-p(k+1)}a_{i+2k-p(k+1)}\cdots a_{i-3k-p(k+1)}a_{i-2k-p(k+1)}.$$
First, assume there exists an integer $p$ such that $0\leq p \leq q$ and $i+k-p(k+1) \equiv l \mod n$ for some
$0\leq l \leq n-k$. Let $p$ be the smallest such number. Then
\begin{equation*}
\hspace{-5 pt}
\begin{array}{r@{}l}
&{} e_{i+k-p(k+1), v c_{i} \cdots c_{i-(p-1)(k+1)},q-p}( w_1)  \vspace{3 pt}   \\
&{}=e_{i+k-p(k+1), v c_{i} \cdots c_{i-(p-1)(k+1)},q-p}(a_1  \cdots a_n v c_{i} \cdots c_{i-(p-1)(k+1)} c_{i-p(k+1)} \cdots a_{i-k-q(k+1)}) \vspace{3 pt}  \\
&{}= a_1 a_2 \cdots a_n v c_{i} \cdots c_{i-(p-1)(k+1)}
a_{i+1-p(k+1)} b_{i+1-p(k+1)} \cdots b_{i+1-q(k+1)}  \vspace{3 pt}   \\
&{}= w_3.
\end{array}
\end{equation*}
So $w_1 \geq w_3$. In case $p=0$, clearly $w_3 = w_2$ and thus $w_2 \geq w_3$.
So the result follows.
In case $p\neq 0$, by the minimality condition, for all
$0\leq p' <p$ we have $i+k-p'(k+1) \equiv l \mod n$ for some
$n-k+1 \leq l \leq n-1 $.
Since $b_ja_{j-1}=a_{j+1}b_{j+1}$ and
$c_{j}a_{j+1-(k+1)}=a_{j+k}c_{j+k}$, we can successively apply the
transformations $d_{i,v,0}$, $d_{i-(k+1),vc_{i},0}$,
$d_{i-2(k+1),vc_{i}c_{i-(k+1)},0},\ldots ,
d_{i-(p-1)(k+1),vc_{i}c_{i-(k+1)} \cdots c_{i-(p-2)(k+1)} ,0}$, obtaining
\begin{equation}\label{successiv}
\begin{array}{r@{}l}
&{}d_{i-(p-1)(k+1),vc_{i}c_{i-(k+1)} \cdots c_{i-(p-2)(k+1)} ,0} \cdots d_{i-2(k+1), vc_{i}c_{i-(k+1)},0} d_{i-(k+1),vc_{i},0} d_{i,v,0} (w_2) \vspace{3 pt} \\
&{}=d_{i-(p-1)(k+1),vc_{i}c_{i-(k+1)} \cdots c_{i-(p-2)(k+1)} ,0} \cdots d_{i-2(k+1), vc_{i}c_{i-(k+1)},0} d_{i-(k+1),vc_{i},0} d_{i,v,0}  \vspace{3 pt} \\
&{}\hspace{17 pt} (a_1 a_2 \cdots a_n v a_{i+1}b_{i+1}b_{i+1-(k+1)} \cdots b_{i+1-q(k+1)}) \vspace{3 pt}  \\
&{}=d_{i-(p-1)(k+1),vc_{i}c_{i-(k+1)} \cdots c_{i-(p-2)(k+1)} ,0} \cdots d_{i-2(k+1), vc_{i}c_{i-(k+1)},0} d_{i-(k+1),vc_{i},0} d_{i,v,0} \vspace{3 pt}  \\
&{}\hspace{17 pt} (a_1 a_2 \cdots a_n v b_{i}a_{i-1}b_{i+1-(k+1)} \cdots b_{i+1-q(k+1)}) \vspace{3 pt}  \\
&{}=d_{i-(p-1)(k+1),vc_{i}c_{i-(k+1)} \cdots c_{i-(p-2)(k+1)} ,0} \cdots d_{i-2(k+1), vc_{i}c_{i-(k+1)},0} d_{i-(k+1),vc_{i},0}  \vspace{3 pt}  \\
&{}\hspace{17 pt} (a_1 a_2 \cdots a_n v a_{i+k}c_{i+k}b_{i+1-(k+1)} \cdots b_{i+1-q(k+1)})  \vspace{3 pt} \\
&{}=d_{i-(p-1)(k+1),vc_{i}c_{i-(k+1)} \cdots c_{i-(p-2)(k+1)} ,0} \cdots d_{i-2(k+1), vc_{i}c_{i-(k+1)},0} d_{i-(k+1),vc_{i},0}  \vspace{3 pt}  \\
&{}\hspace{17 pt} (a_1 a_2 \cdots a_n v c_{i}a_{i-k}b_{i+1-(k+1)} \cdots b_{i+1-q(k+1)})  \vspace{3 pt} \\
&{}=d_{i-(p-1)(k+1),vc_{i}c_{i-(k+1)} \cdots c_{i-(p-2)(k+1)} ,0} \cdots d_{i-2(k+1), vc_{i}c_{i-(k+1)},0} d_{i-(k+1),vc_{i},0}   \vspace{3 pt} \\
&{}\hspace{17 pt} (a_1 a_2 \cdots a_n v c_{i}a_{i+1-(k+1)}b_{i+1-(k+1)} \cdots b_{i+1-q(k+1)}) \vspace{3 pt}  \\
&{} \hspace{6 pt} \vdots \vspace{3 pt}  \\
&{}= a_1\cdots a_n v  c_{i}\cdots c_{i-(p-2)(k+1)} a_{i+k-(p-1)(k+1)}  c_{i+k-(p-1)(k+1)}   b_{i+1-p(k+1)} \cdots b_{i+1-q(k+1)} \vspace{3pt} \\
&{}= a_1\cdots a_n v  c_{i}\cdots c_{i-(p-2)(k+1)} c_{i-(p-1)(k+1)}a_{i-k-(p-1)(k+1)}  b_{i+1-p(k+1)} \cdots b_{i+1-q(k+1)} \vspace{3pt} \\
&{}= a_1\cdots a_n v  c_{i}\cdots c_{i-(p-2)(k+1)}
c_{i-(p-1)(k+1)}a_{i+1-p(k+1)} b_{i+1-p(k+1)} \cdots
b_{i+1-q(k+1)} \vspace{3pt} \\
&{} =w_3.
\end{array}
\end{equation}
So $w_2 \geq w_3$ and we have proved the statement.

Second, assume no $p$ as above exists. One can now apply the calculation in (\ref{successiv}) for the number $p=q+1$ and this
shows that $w_2 \geq w_3$. It it is easily seen that in this case $w_3= w_1$ and thus the result is proved.
\end{proof}

\begin{theorem}\label{normalform}
Each element $a\in S_n$ can be uniquely written as a product
\begin{eqnarray}\label{nf}
a=a_1^{i}(a_2\cdots a_n)^jb,
\end{eqnarray}
where   $b\in S_n\setminus (a_1S_n\cup a_2\cdots a_nS_n )$, $i,j$
are non-negative integers such that if $i,j>0$ then
$$\begin{array}{r@{}l}
b \notin  \bigcup_{q \geq 0}  \Big(
&{} \bigcup_{i=n-k+1}^{n-1 }  S_n  b_{i}b_{i-(k+1)} \cdots b_{i-q(k+1)} a_{i-1-q(k+1)}  S_n \vspace{3pt} \\
&{}\cup \bigcup_{i=0}^{n-k } S_n     c_{i}c_{i-(k+1)}
\cdots c_{i-q(k+1)} a_{i-k-q(k+1)}  S_n  \vspace{3pt} \\
&{}\cup \bigcup_{i=1}^n \big(  S_n   a_i
b_{i}b_{i-(k+1)} \cdots b_{i-q(k+1)} a_{i-1-q(k+1)}      \vspace{3pt} \\
&{} \hspace{42 pt} \cup  S_n  a_i c_{i}c_{i-(k+1)}\cdots c_{i-q(k+1)} a_{i-k-q(k+1)} S_n \big)
\Big)
\end{array}$$
\end{theorem}

\begin{proof}
We will apply the diamond lemma
\cite[Theorem~10.4.1]{cohn}  
to the reduction system\index{reduction system} $T=\{(w_\varphi, y_\varphi)\}$,
where $w_\varphi,y_\varphi \in \FM_n$ are such that  $w_\varphi$ gets mapped to \index{mapped to} $y_\varphi$ by
one of the transformations\index{transformation} (\ref{ti})-(\ref{ui}).
The partial order $\leq$ defined above is compatible\index{compatible} with the structure\index{structure} of $\FM_n$ and
satisfies the descending chain condition\index{descending chain condition}.
It is clear that $y_\varphi < w_\varphi$ for every element of $T$,
thus $T$ is reduction-finite\index{reduction-finite}.
To prove the result, we need to verify that all ambiguities\index{ambiguity} of $T$ are resolvable\index{resolvable}.

Let $w,w_1,w_2\in \FM_n$ such that $w\succ w_1$ and $w\succ
w_2$. We should show that there exists $w_{3}\in \FM_{n}$ such that
$w_{1},w_{2}\geq w_{3}$. We know that $w=a_{i_1}a_{i_2}\cdots
a_{i_m}$,
$$w_1=a_{i_1}\cdots a_{i_x}f(a_{i_{x+1}}\cdots
a_{i_{x+y}})a_{i_{x+y+1}}\cdots a_{i_m}$$ and
$$w_2=a_{i_1}\cdots a_{i_l}g(a_{i_{l+1}}\cdots
a_{i_{l+p}})a_{i_{l+p+1}}\cdots a_{i_m},$$ where $a_{i_1},\dots
,a_{i_m}\in \{a_1,\dots ,a_n\}$ and
\begin{equation*}
\begin{array}{r@{}l}
f,g &{} \in \{t_i\mid 2\leq i\leq
n\} \cup \{r_{j,m}\mid 2\leq j\leq n \text{ and } m\geq 1\}\cup
\{h_i\mid 1\leq i\leq n\}  \vspace{3 pt} \\
&{} \hspace{17 pt} \cup   \{d_{i,v,q}\mid v\in\FM_n,\, n-k+1 \leq i \leq n-1 \text{ and } q
\geq 0 \} \vspace{3 pt}  \\
&{} \hspace{17 pt} \cup \{e_{i,v,q}\mid v\in\FM_n,\, 0 \leq i \leq n-k \text{
and } q \geq 0  \}  \vspace{3 pt} \\
&{} \hspace{17 pt} \cup \{s_{i,v,q} \mid v\in\FM_n, 1 \leq i \leq
n \text{ and } q \geq 0 \}  \vspace{3 pt} \\
&{} \hspace{17 pt} \cup \{u_{i,v,q} \mid v\in\FM_n,\, 1
\leq i \leq n \text{ and } q \geq 0 \}.
\end{array}
\end{equation*}
Note that if $x+y<l+1$ or $l+p<x+1$, then applying $g$ to the
subword\index{subword} $a_{i_{l+1}}\cdots a_{i_{l+p}}$ of $w_1$ and $f$ to the
subword $a_{i_{x+1}}\cdots a_{i_{x+y}}$ of $w_2$ we obtain the same
word $w_3$ and $$w_1\succ w_3\quad\mbox{and}\quad w_2\succ w_3.$$

If $x+y \geq l+1$ and $l+p\geq x+1$, we say that the subwords
$a_{i_{x+1}}\cdots a_{i_{x+y}}$ and $a_{i_{l+1}}\cdots a_{i_{l+p}}$
of $w$ overlap. We will  study all possible overlaps\index{overlap} between
subwords\index{subword} of the forms
\begin{itemize}
\item[$(\alpha)$] $a_{i} a_{i+1} \cdots a_{i-2} a_{i-1}$, \hspace{12 pt} for $2\leq
i\leq n$,
\item[$(\beta)$] $a_{i}a_1^{m}a_2\cdots a_n$, \hspace{12 pt} for $2\leq i\leq n$
and $m\geq 1$,
\item[$(\delta)$] $a_i a_{i+k} a_{i+2k} \cdots  a_{i-2k} a_{i- k}$,
\hspace{12 pt} for $i=1,2,\dots ,n$,
\item[$(\zeta)$] $a_1 \cdots a_n v  b_{i}b_{i-(k+1)} \cdots b_{i-q(k+1)}
a_{i-1-q(k+1)}$, \vspace{3 pt} \\
  for $v\in \FM_n$, $n-k+1 \leq i
\leq n-1 $ and  $q\geq 0$,
\item[$(\eta)$] $a_1 \cdots a_n v  c_{i}c_{i-(k+1)} \cdots c_{i-q(k+1)}
a_{i-k-q(k+1)}$, \vspace{3 pt} \\
for $v\in \FM_n$, $  0 \leq i \leq
n-k $ and  $q\geq 0$,
\item[$(\theta)$]
$ a_1 \cdots a_n v a_i b_{i}b_{i-(k+1)} \cdots b_{i-q(k+1)}
a_{i-1-q(k+1)} $, \vspace{3 pt} \\
for $v\in \FM_n$, $  1 \leq i \leq
n $ and  $q\geq 0$,
\item[$(\iota)$]
$a_1 \cdots a_n v a_i c_{i}c_{i-(k+1)}\cdots c_{i-q(k+1)}
a_{i-k-q(k+1)}$, \vspace{3 pt} \\
for $v\in \FM_n$, $1 \leq i \leq
n $ and  $q\geq 0$.
\end{itemize}

Note that overlaps\index{overlap} of subwords\index{subword} of the form $(\alpha)$ and $(\beta)$
are dealt with in \cite[Theorem~2.1]{alghomshort}. Note
also that not all kinds of subwords can overlap, and that for
overlaps with the subword $a_1 \cdots a_n$ at the beginning of words
of type $(\zeta)$,$(\eta)$,$(\theta)$,$(\iota)$, the verification
that the ambiguity\index{ambiguity} is resolvable\index{resolvable} is immediate.

The remaining cases to consider are overlaps of the following form,
the pairs in square brackets can be
verified in a manner analogous to the pair preceding the bracket: \\
$(\alpha ,\delta )$, [$(\delta, \alpha)$];   \ \ $(\beta, \delta )$;
\ \ $(\delta, \beta )$; \ \ $(\delta, \delta)$; \ \ $(\zeta,
\alpha)$, [$(\eta, \delta)$]; \  \ $(\zeta, \beta)$, [$(\eta,
\beta)$, $(\theta, \beta)$, $(\iota, \beta)$] ; \\
$(\zeta, \delta )$, [$(\eta, \alpha )$]; \ \
 $(\theta,\alpha)$, [$(\iota, \delta)$]; \   \ $(\theta, \delta)$, [$(\iota, \alpha)$];\\
where the notation $(\diamondsuit , \heartsuit )\label{suitss1}$ means that the
ending letter(s) of a subword of the form $\diamondsuit$, overlap(s)
with the beginning letter(s) of a subword of the form $\heartsuit$,
as well as

$(\alpha, \zeta)$, [$(\delta, \eta)$]; \  \ $(\alpha, \eta)$,
[$(\delta, \zeta)$];  \  \
$(\alpha, \theta)$, [$(\delta, \iota)$]; \  \  $(\alpha, \iota)$, [$(\delta, \theta)$];   \\
$(\beta, \zeta)$; \  \ $(\beta, \eta)$; \  \ $(\beta, \theta)$; \  \  $(\beta, \iota)$; \\
$(\zeta, \zeta)$, [$(\eta, \eta)$]; \ \  $(\zeta, \eta)$, [$(\eta, \zeta)$]; \ \  $(\zeta, \theta)$, [$(\eta, \iota)$]; \  \  $(\zeta, \iota)$, [$(\eta, \theta)$]; \\
$(\theta, \zeta)$, [$(\iota, \eta)$]; \  \ $(\theta, \eta)$, [$(\iota, \zeta)$];  \  \  $(\theta, \theta)$, [$(\iota, \iota)$]; \  \  $(\theta, \iota)$, [$(\iota, \theta)$]; \\
where the notation $(\spadesuit , \clubsuit )\label{suitss2}$ means that the ending
letter(s) of a subword of the form $\spadesuit$, overlap(s) with the
beginning letter(s) of a subword beginning just after the end of the
subword $v$ in a subword of the form $\clubsuit $.
\vspace{18 pt} \\
{\it Case 1:}
Overlaps of the form $(\alpha ,\delta )$, [$(\delta, \alpha)$].
 We apply $t_i$ to the subword of type $\alpha$ to get $w_1$
and we apply $h_{i-1}$ to the subword of type $\delta $ to get $w_2$.  We have
$$ w=  \overbrace{
a_{i} a_{i+1} \cdots a_{i-2}  
\makebox[0pt][l]{$\displaystyle{\underbrace{\phantom{
a_{i-1} a_{i-1+k} a_{i-1+2k} \cdots  a_{i-1-2k} a_{i-1- k}   
} }_{\text{$\delta$, $ h_{i-1}$: $w_2$ } } } $}   
a_{i-1}                                                  
}^{\text{$\alpha$, $t_{i}$: $w_1$}}    
a_{i-1+k} a_{i-1+2k} \cdots  a_{i-1-2k} a_{i-1- k} . $$   
$$w_1=a_1 a_2 \cdots a_n a_{i-1+k} a_{i-1+2k} \cdots  a_{i-1-2k} a_{i-1- k}$$  and
$$w_2=a_{i} a_{i+1} \cdots a_{i-2} a_1 a_2 \cdots  a_n   .$$
Applying a sequence of transformations $r$, we get
\begin{eqnarray*}
w_2&=&a_{i} a_{i+1} \cdots a_{i-2} a_1 a_2 \cdots  a_n  \succ \ldots \succ a_1 a_2 \cdots  a_n a_{i} a_{i+1} \cdots a_{i-2}  = w_2'.
\end{eqnarray*}
Then either $d_{i-1,1,0}(w_2')=w_1$ or $e_{i-1,1,0}(w_1)=w_2'$.
\vspace{18 pt} \\
{\it Case 2:} Overlaps of the form $(\beta, \delta )$.
We apply $r_{j,m}$ to the subword of type $\beta$ to get
$w_1$ and we apply $h_n$ to the subword of type $\delta$ to get
$w_2$.  We have
$$ w=  \overbrace{
a_{j}a_1^{m}a_2\cdots   
\makebox[0pt][l]{$\displaystyle{\underbrace{\phantom{
a_n a_{k} a_{2 k} \cdots  a_{-2 k} a_{- k}   
} }_{\text{$\delta$, $h_n$: $w_2$ } } } $}   
a_n                                         
}^{\text{$\beta$, $r_{j,m}$: $w_1$}}    
a_{k} a_{2 k} \cdots  a_{-2 k} a_{- k} .$$   
$w_1=a_1\cdots a_n a_{j}a_1^{m-1} a_{k} a_{2 k} \cdots  a_{-2 k} a_{- k}     \quad\mbox{and}\quad
w_2=a_{j}a_1^{m}a_2\cdots a_{n-1} a_1 \cdots a_n.$
\vspace{6 pt} \\
Applying $e_{n,a_{j}a_1^{m-1},0}$, we get
$$  w_1= a_1\cdots a_n a_{j}a_1^{m-1} a_{k} a_{2 k} \cdots  a_{-2 k} a_{- k} \succ a_1 \cdots a_n a_{j}a_1^{m-1}a_1a_2\cdots a_{-2} a_{-1}. $$
Applying a sequence of transformations $r$, we get
$$ w_2= a_{j}a_1^{m}a_2\cdots a_{n-1} a_1 \cdots a_n \succ \ldots \succ a_1 \cdots a_n a_{j}a_1^{m}a_2\cdots a_{n-1}.$$
\vspace{6 pt} \\
{\it Case 3:} Overlaps of the form $(\delta,\beta)$. Here there are two possibilities:
\begin{itemize}
\item[(a)] $a_{i-k}=a_1$. We apply $h_{1+k}$ to the subword of type $\delta$ to get $w_1$
and we apply $r_{1-k,m}$ to the subword subword of type $\beta$ to get $w_2$.
$$ w=  \overbrace{
a_{1+k}a_{1+2k}\cdots   
\makebox[0pt][l]{$\displaystyle{\underbrace{\phantom{
a_{1-k} a_{1} a_{1}^{m-1}a_2 \cdots  a_{n}   
} }_{\text{$\beta$, $r_{1-k,m}$: $w_2$ } } } $}   
a_{1-k}a_1                                         
}^{\text{$\delta$, $h_{1+k}$: $w_1$}}    
a_{1}^{m-1}a_2 \cdots  a_{n} .$$   
$w_1=a_1\cdots a_na_{1}^{m-1}a_2 \cdots  a_{n}     \quad\mbox{and}\quad
w_2= a_{1+k}a_{1+2k}\cdots a_{1-2k} a_1\cdots a_n  a_{1-k} a_1^{m-1} .$

If $m \neq 1$, applying a sequence of transformations $r$, we get
$$ w_1 = a_1\cdots a_na_{1}^{m-1}a_2 \cdots  a_{n} \succ \ldots \succ a_1^2 (a_2 \cdots a_n)^2 a_1^{m-2}.$$
Applying a sequence of transformations $r$, followed by $h_{1+k}$, followed by a sequence of transformations $r$, we get
\begin{equation*}
\begin{array}{r@{}l}
w_2&{}=  a_{1+k}a_{1+2k}\cdots a_{1-2k} a_1\cdots a_n  a_{1-k} a_1^{m-1}  \\
&{} \succ \ldots \succ
a_1\cdots a_n  a_{1+k}a_{1+2k}\cdots a_{1-2k}  a_{1-k} a_1^{m-1} \\
&{} \succ  a_1\cdots a_n a_1 \cdots a_n a_1^{m-2} \succ \ldots \succ  a_1^2 (a_2 \cdots a_n)^2 a_1^{m-2}.
\end{array}
\end{equation*}
If $m = 1$, applying a sequence of transformations $r$ followed by $e_{1,1,0}$, we get
\begin{equation*}
\begin{array}{r@{}l}
w_2&{}=  a_{1+k}a_{1+2k}\cdots a_{1-2k} a_1\cdots a_n  a_{1-k} \succ \ldots \succ a_1 \cdots a_n a_{1+k}a_{1+2k}\cdots a_{1-2k} a_{1-k} \\
&{}\succ a_1 \cdots a_n a_2 \cdots a_n = w_1.
\end{array}
\end{equation*}
\item[(b)] $a_{i-k}\neq a_1$. We apply $h_i$ to the subword of type $\delta$ to get $w_1$
and we apply $r_{i-k,m}$ to the subword of type $\beta$ to get $w_2$.
$$ w=  \overbrace{
a_{i}a_{i+k}\cdots  a_{i-2k} 
\makebox[0pt][l]{$\displaystyle{\underbrace{\phantom{
a_{i-k}a_{1}^{m}a_2 \cdots  a_{n}   
} }_{\text{$\beta$, $r_{i-k,m}$: $w_2$ } } } $}   
a_{i-k}                                         
}^{\text{$\delta$, $h_{i}$: $w_1$}}    
a_{1}^{m}a_2 \cdots  a_{n} .$$   
\end{itemize}
$w_1=a_1\cdots a_na_{1}^{m}a_2 \cdots  a_{n}     \quad\mbox{and}\quad
w_2= a_{i}a_{i+k}\cdots a_{i-2k} a_1\cdots a_n  a_{i-k} a_1^{m-1} .$
Applying a sequence of transformations $r$, we get
$$ w_1 = a_1\cdots a_na_{1}^{m}a_2 \cdots  a_{n} \succ \ldots \succ a_1^2 (a_2 \cdots a_n)^2 a_1^{m-1}.$$
If $i\neq 1$, applying a sequence of transformations $r$, followed by $h_{i}$, followed by a sequence of transformations $r$, we get
\begin{equation*}
\begin{array}{r@{}l}
w_2&{} =   a_{i}a_{i+k}\cdots a_{i-2k} a_1\cdots a_n  a_{i-k} a_1^{m-1}  \succ \ldots \succ
a_1\cdots a_n  a_{i}a_{i+k}\cdots a_{i-2k}  a_{i-k} a_1^{m-1} \\
&{} \succ  a_1\cdots a_n a_1 \cdots a_n a_1^{m-1} \succ  a_1^2 (a_2 \cdots a_n)^2 a_1^{m-1}.
\end{array}
\end{equation*}
If $i= 1$, applying a sequence of transformations $r$, followed by $e_{1,1,0}$, we get
\begin{equation*}
\begin{array}{r@{}l}
w_2&{} =   a_{1}a_{1+k}\cdots a_{1-2k} a_1\cdots a_n  a_{1-k} a_1^{m-1}  \succ \ldots \succ
a_1^2 a_2\cdots a_n  a_{1+k}\cdots a_{1-2k}  a_{1-k} a_1^{m-1} \\
&{} \succ  a_1^2 a_2 \cdots a_n  a_2 \cdots a_n a_1^{m-1} .
\end{array}
\end{equation*}
\vspace{0 pt} \\
{\it Case 4:} Overlaps of the form $(\delta, \delta )$. We apply
$h_i$ to a subword of type $\delta$ to get $w_1$ and we apply
$h_{i'}$ to the other subword of type $\delta$ to get $w_2$.  We
have
$$ w=  \overbrace{
a_i a_{i+k}  \cdots   
\makebox[0pt][l]{$\displaystyle{\underbrace{\phantom{
a_{i'} \cdots a_{i- k}   \cdots a_{i'- k}  
} }_{\text{$\delta$, $h_{i'}$: $w_2$ } } } $}   
a_{i'} \cdots a_{i- k}                                     
}^{\text{$\delta$, $h_i$: $w_1$}}    
  \cdots a_{i'- k} .$$   
$w_1=a_1   \cdots a_n a_{i}   \cdots a_{i'- k} \quad\mbox{and}\quad  w_2=a_i  \cdots a_{i'-k} a_1 \cdots a_n.$
\vspace{6 pt} \\
If $i\neq 1$, applying a sequence of transformations $r$, we get
$$ w_2= a_i  \cdots a_{i'-k} a_1 \cdots a_n \succ \ldots \succ a_1   \cdots a_n a_{i}   \cdots a_{i'- k} =w_1.$$
If $i = 1$, applying $t_2$, we get
$$ w_1=a_1   \cdots a_n a_{1} a_{1+k}   \cdots a_{i'- k}   \succ   a_1^2 a_2  \cdots a_n a_{1+k}   \cdots a_{i'- k},$$
and applying a sequence of transformations $r$, we get
$$w_2=a_1  \cdots a_{i'-k} a_1 \cdots a_n \succ  \ldots \succ   a_1^2 a_2  \cdots a_n a_{1+k}   \cdots a_{i'- k}.$$
{\it Case 5:} Overlaps of the form $(\zeta, \alpha )$ ,[$(\eta,
\delta)$]. We apply $d_{i,v,q}$ to the subword of type $\zeta$ to
get $w_1$ and we apply $t_{i'}$ to the subword of type $\alpha$ to
get $w_2$.  We have
$$ w=  \overbrace{
 a_1 \cdots a_n v  b_{i}b_{i-(k+1)} \cdots b_{i-(q-1)(k+1)} a_{i+1-q(k+1)}    \cdots   
\makebox[0pt][l]{$\displaystyle{\underbrace{\phantom{
 a_{i'} \cdots  a_{i-1-q(k+1)} \cdots   a_{i'-1} 
} }_{\text{$\alpha$, $t_{i'}$: $w_2$ } } } $}   
a_{i'} \cdots  a_{i-1-q(k+1)}                                 
}^{\text{$\zeta$, $d_{i,v,q}$: $w_1$}}    
 \cdots   a_{i'-1} .$$   
\begin{equation*}
\begin{array}{r@{}l}
w_1&{}=a_1 \cdots a_n v  a_{i+k} c_{i+k} c_{i+k-(k+1)} \cdots c_{i+k-q(k+1)} a_{i-q(k+1)} \cdots   a_{i'-1}, \vspace{3 pt}  \\
w_2&{}= a_1 \cdots a_n v   b_{i}b_{i-(k+1)} \cdots b_{i-(q-1)(k+1)} a_{i+1-q(k+1)} \cdots a_{i'-1} a_1 \cdots a_n.
\end{array}
\end{equation*}
Applying $u_{i+k,v,q}$, we get
\begin{equation*}
\begin{array}{r@{}l}
w_1 &{}=  a_1 \cdots a_n v  a_{i+k} c_{i+k} c_{i+k-(k+1)} \cdots c_{i+k-q(k+1)} a_{i+k-k-q(k+1)}a_{i+1-q(k+1)}\cdots   a_{i'-1} \vspace{3 pt} \\
&{} \succ  a_1^2 (a_2 \cdots a_n)^2  v b_{i}b_{i-(k+1)} \cdots
b_{i-(q-1)(k+1)} a_{i+1-q(k+1)} \cdots a_{i'-1}.
\end{array}
\end{equation*}
Applying a sequence of transformations $r$, we get
\begin{equation*}
\begin{array}{r@{}l}
w_2  &{}= a_1 \cdots a_n v   b_{i}b_{i-(k+1)} \cdots b_{i-(q-1)(k+1)} a_{i+1-q(k+1)} \cdots a_{i'-1} a_1 \cdots a_n \vspace{3 pt}   \\
&{} \succ   \dots \succ a_1^2 (a_2 \cdots a_n)^2  b_{i}b_{i-(k+1)}
\cdots b_{i-(q-1)(k+1)} a_{i+1-q(k+1)} \cdots a_{i'-1}.
\end{array}
\end{equation*}
\vspace{0 pt}\\
{\it Case 6:} Overlaps of the form $(\zeta, \beta)$, [$(\eta,
\beta)$, $(\theta, \beta)$, $(\iota, \beta)$]. We apply
$d_{i,v,q}$ to the subword of type $\zeta$ to get $w_1$ and we apply
$r_{j,m}$ to the subword of type $\beta$ to get $w_2$. We have
$j=i-1-q(k+1)$. There are two possibilities.  \vspace{ 6 pt} \\
(a)   $j=i-1-q(k+1)\neq 1$.
$$ w=  \overbrace{
a_1 \cdots a_n v  b_{i} b_{i-(k+1)} \cdots b_{i-q(k+1)} 
\makebox[0pt][l]{$\displaystyle{\underbrace{\phantom{
 a_{i-1-q(k+1)}a_1^{m}a_2\cdots a_n   
} }_{\text{$\beta$, $r_{j,m}$: $w_2$ } } } $}   
a_{i-1-q(k+1)}                                 
}^{\text{$\zeta$, $d_{i,v,q}$: $w_1$}}    
a_1^{m}a_2\cdots a_n .$$   
\begin{equation*}
\begin{array}{r@{}l}
w_1&{}=a_1 \cdots a_n v  a_{i+k}  c_{i+k} c_{i+k-(k+1)} \cdots c_{i+k-q(k+1)} a_1^{m}a_2\cdots a_n, \vspace{3 pt} \\
w_2&{}=a_1 \cdots a_n v  b_{i} b_{i-(k+1)} \cdots b_{i-q(k+1)} a_1
a_2\cdots a_n  a_{i-1-q(k+1)} a_1^{m-1}.
\end{array}
\end{equation*}
Applying a sequence of transformations $r$, we get
\begin{equation*}
\begin{array}{r@{}l}
w_1&{}=a_1 \cdots a_n v  a_{i+k}  c_{i+k}
c_{i+k-(k+1)} \cdots c_{i+k-q(k+1)} a_1^{m}a_2\cdots a_n \vspace{3 pt} \\
&{}\succ \ldots \succ  a_1^2 (a_2 \cdots a_n)^2 v a_{i+k}  c_{i+k}
c_{i+k-(k+1)} \cdots c_{i+k-q(k+1)} a_1^{m-1}.
\end{array}
\end{equation*}
Applying a sequence of transformations $r$ and $d_{i,a_2\cdots
a_nv,q}$, we get
\begin{equation*}
\begin{array}{r@{}l}
w_2 &{} = a_1 \cdots a_n v  b_{i} b_{i-(k+1)} \cdots b_{i-q(k+1)} a_1 a_2\cdots a_n  a_{i-1-q(k+1)} a_1^{m-1} \vspace{3 pt} \\
&{} \succ   \ldots \succ  a_1^2 (a_2 \cdots a_n)^2 v  b_{i}
b_{i-(k+1)} \cdots b_{i-q(k+1)} a_{i-1-q(k+1)} a_1^{m-1} \vspace{3 pt} \\
&{} \succ   a_1^2 (a_2 \cdots a_n)^2 v  a_{i+k}  c_{i+k}
c_{i+k-(k+1)} \cdots c_{i+k-q(k+1)} a_1^{m-1}.
\end{array}
\end{equation*}
(b) $j=i-1-q(k+1)=1$.
$$ w=  \overbrace{
a_1 \cdots a_n v  b_{i} b_{i-(k+1)} \cdots b_{i-q(k+1)} 
\makebox[0pt][l]{$\displaystyle{\underbrace{\phantom{
 a_{1}a_1^{m-1}a_2\cdots a_n   
} }_{\text{$\beta$, $r_{j,m}$: $w_2$ } } } $}   
a_{1}                                 
}^{\text{$\zeta$, $d_{i,v,q}$: $w_1$}}    
a_1^{m-1}a_2\cdots a_n .$$   
\begin{equation*}
\begin{array}{r@{}l}
w_1&{}=a_1 \cdots a_n v  a_{i+k}  c_{i+k} c_{i+k-(k+1)} \cdots c_{i+k-q(k+1)} a_1^{m-1}a_2\cdots a_n, \vspace{3 pt} \\
w_2&{}=a_1 \cdots a_n v  b_{i} b_{i-(k+1)} \cdots b_{i-(q-1)(k+1)} a_{i+1-q(k+1)} \cdots a_{i-3-q(k+1)} \cdot  \vspace{3 pt} \\
&{}   \hspace{ 250 pt} \cdot a_1a_2\cdots a_n  a_{i-2-q(k+1)} a_1^{m-1}.
\end{array}
\end{equation*}
If $m\neq 1$, this is solved in the same way as (a). If $m=1$, we use the facts that $a_{i+k-k-q(k+1)}=a_2$ and
$a_3 \cdots a_n= a_{i+1-q(k+1)} \cdots a_{i-2-q(k+1)} = b_{i-q(k+1)}$ and applying $u_{i+k,v,q}$, we get
\begin{equation*}
\begin{array}{r@{}l}
w_1&{}= a_1 \cdots a_n v  a_{i+k}  c_{i+k} c_{i+k-(k+1)} \cdots c_{i+k-q(k+1)} a_2\cdots a_n  \vspace{3 pt}  \\
&{}\succ  a_1^2 (a_2 \cdots a_n)^2 v   b_{i+k+1-(k+1)} b_{i+k+1-2(k+1)} \cdots b_{i+k+1-q(k+1)} a_3\cdots a_n  \vspace{3 pt} \\
&{}= a_1^2 (a_2 \cdots a_n)^2 v   b_{i} b_{i-(k+1)} \cdots b_{i-(q-1)(k+1)} b_{i-q(k+1)}.
\end{array}
\end{equation*}
Applying a sequence of transformations $r$, we get
\begin{equation*}
\begin{array}{r@{}l}
w_2&{}=a_1 \cdots a_n v  b_{i} b_{i-(k+1)} \cdots b_{i-(q-1)(k+1)} a_{i+1-q(k+1)} \cdots a_{i-3-q(k+1)} \cdot  \vspace{3 pt} \\
&{}   \hspace{ 275 pt} \cdot a_1a_2\cdots a_n  a_{i-2-q(k+1)}   \vspace{3 pt} \\
&{}\succ  a_1^2 (a_2 \cdots a_n)^2 v  b_{i} b_{i-(k+1)} \cdots b_{i-(q-1)(k+1)} a_{i+1-q(k+1)} \cdots a_{i-3-q(k+1)}  a_{i-2-q(k+1)}  \vspace{3 pt} \\
&{} = a_1^2 (a_2 \cdots a_n)^2 v   b_{i} b_{i-(k+1)} \cdots b_{i-(q-1)(k+1)} b_{i-q(k+1)}.
\end{array}
\end{equation*}
\vspace{0 pt} \\
{\it Case 7:} Overlaps of the form $(\zeta, \delta )$, [$(\eta,
\alpha )$]. We apply $d_{i,v,q}$ to the subword of type $\zeta$ to
get $w_1$ and we apply $h_{i-1-q(k+1)}$ to the subword of type
$\delta$ to get $w_2$.  We have
$$ w=  \overbrace{
 a_1 \cdots a_n v  b_{i}b_{i-(k+1)} \cdots b_{i-q(k+1)}   
\makebox[0pt][l]{$\displaystyle{\underbrace{\phantom{
 a_{i-1-q(1+k)}  a_{i-1-q(k+1)+k}  \cdots  a_{i-1-q(k+1)-k} 
} }_{\text{$\delta$, $h_{i-1-q(k+1)}$: $w_2$ } } } $}   
a_{i-1-q(k+1)}                        
}^{\text{$\zeta$, $d_{i,v,q}$: $w_1$}}    
  a_{i-1-q(k+1)+k}  \cdots  a_{i-1-q(k+1)-k}.  $$   
\begin{equation*}
\begin{array}{r@{}l}
 w_1&{}=a_1 \cdots a_n v  a_{i+k} c_{i+k} c_{i+k-(k+1)} \cdots c_{i+k-q(k+1)} a_{i-1-q(k+1)+k}  \cdots  a_{i-1-q(k+1)-k}, \vspace{3 pt} \\
 w_2&{}= a_1 \cdots a_n v b_{i}b_{i-(k+1)} \cdots b_{i-q(k+1)} a_1\cdots a_n .
\end{array}
\end{equation*}
Applying $u_{i+k,v,q+1}$, we get
\begin{equation*}
\begin{array}{r@{}l}
w_1 &{}= a_1 \cdots a_n v a_{i+k} c_{i+k} c_{i+k-(k+1)} \cdots c_{i+k-q(k+1)} a_{i+k+k-q(k+1)-k-1} \cdots \vspace{3 pt} \\
&{} \hspace{285 pt} \cdots  a_{i+k-k-q(k+1)-k-1}  \\
&{}= a_1 \cdots a_n v a_{i+k} c_{i+k} c_{i+k-(k+1)} \cdots c_{i+k-q(k+1)} c_{i+k-(q+1)(k+1)}  a_{i+k-k-(q+1)(k+1)} \vspace{3 pt} \\
&{} \succ   a_1^2 (a_2 \cdots a_n)^2 v b_{i} b_{i-(k+1)} \cdots
b_{i-q(k+1)}.
\end{array}
\end{equation*}
Applying a sequence of transformations $r$ we get
\begin{equation*}
\begin{array}{r@{}l}
w_2&{} = a_1 \cdots a_n v b_{i}b_{i-(k+1)} \cdots b_{i-q(k+1)} a_1 \cdots a_n  \vspace {3 pt} \\
&{} \succ \ldots \succ a_1^2 (a_2 \cdots a_n)^2 v b_{i}b_{i-(k+1)} \cdots b_{i-q(k+1)}.
\end{array}
\end{equation*}
\vspace{0 pt} \\
{\it Case 8:} Overlaps of the form  $(\theta,\alpha)$, [$(\iota,
\delta)$]. We apply $s_{i,v,q}$ to the subword of type $\theta$ to
get $w_1$, and we apply $t_{i'}$ to the subword of type $\alpha$ to
get $w_2$. We have
$$ w=  \overbrace{
a_1 \cdots a_n v a_{i}   b_{i}b_{i-(k+1)} \cdots b_{i-(q-1)(k+1)} a_{i+1-q(k+1)}  \cdots   
\makebox[0pt][l]{$\displaystyle{\underbrace{\phantom{
a_{i'} \cdots a_{i-1 -q(k+1)}   \cdots a_{i'-1}  
} }_{\text{$\alpha$, $t_{i'}$: $w_2$ } } } $}   
a_{i'}\cdots a_{i-1 -q(k+1)}                                
}^{\text{$\theta$, $s_{i,v,q}$: $w_1$}}    
 \cdots a_{i'-1} .$$   
 \begin{equation*}
\begin{array}{r@{}l}
w_1&{}=  a_1^2 (a_2 \cdots a_n)^2  v c_{i+k-(k+1)} c_{i+k-2(k+1)} \cdots c_{i+k-q(k+1)}  a_{i -q(k+1)}   \cdots a_{i'-1}, \vspace{3 pt} \\
w_2&{}=  a_1 \cdots a_n v  a_{i}  b_{i}b_{i-(k+1)} \cdots b_{i-(q-1)(k+1)} a_{i+1-q(k+1)}  \cdots a_{i'-1}a_1    \cdots a_n .
\end{array}
\end{equation*}
We have
\begin{equation*}
\begin{array}{r@{}l}
w_1&{}=  a_1^2 (a_2 \cdots a_n)^2  v c_{i+k-(k+1)} c_{i+k-2(k+1)} \cdots c_{i+k-q(k+1)}  a_{i -q(k+1)}   \cdots a_{i'-1} \vspace{3 pt} \\
&{}= a_1^2 (a_2 \cdots a_n)^2  v c_{i-1}c_{i-1-(k+1)} \cdots
c_{i-1-(q-1)(k+1)} a_{i-1-k-(q-1)(k+1)} \cdots a_{i'-1}.
\end{array}
\end{equation*}
Applying a sequence of transformations $r$, we get
$$ w_2 \succ \ldots \succ a_1^2 (a_2 \cdots a_n)^2 v  a_{i}  b_{i}b_{i-(k+1)} \cdots b_{i-(q-1)(k+1)} a_{i+1-q(k+1)}  \cdots a_{i'-1}= w_2'.$$
By Lemma~\ref{joinn}, there exists $w_3\in \FM_n$ such that
$w_1,w_2'\geq w_3$.
\vspace{18 pt} \\
{\it Case 9:} Overlaps of the form  $(\theta, \delta)$, [$(\iota,
\alpha)$]. We apply $s_{i,v,q}$ to the subword of type $\theta$ to
get $w_1$, and we apply $h_{i-1-q(k+1)}$ to the subword of type
$\delta$ to get $w_2$. We have
$$ w=  \overbrace{
 a_1 \cdots a_n v  a_{i}  b_{i} b_{i-(k+1)} \cdots b_{i-q(k+1)}   
\makebox[0pt][l]{$\displaystyle{\underbrace{\phantom{
 a_{i-1 -q(k+1)} a_{i-1 -q(k+1)+k} \cdots a_{i-1 -q(k+1)-k}  
} }_{\text{$\delta$, $h_{i-1-q(k+1)}$: $w_2$ } } } $}   
 a_{i-1 -q(k+1)}                           
}^{\text{$\theta$, $s_{i,v,q}$: $w_1$}}    
 a_{i-1 -q(k+1)+k} \cdots a_{i-1 -q(k+1)-k} .$$   
\begin{equation*}
\begin{array}{r@{}l}
w_1&{}=  a_1^2 (a_2 \cdots a_n)^2  v c_{i+k-(k+1)} c_{i+k-2(k+1)} \cdots c_{i+k-q(k+1)}  a_{i-1 -q(k+1)+k} \cdots a_{i-1 -q(k+1)-k}, \vspace{3 pt} \\
 w_2&{}=  a_1 \cdots a_n v  a_{i} b_{i} b_{i-(k+1)} \cdots b_{i-q(k+1)} a_1 \cdots a_n.
\end{array}
\end{equation*}
We have
\begin{equation*}
\begin{array}{r@{}l}
w_1 &{}= a_1^2 (a_2 \cdots a_n)^2  v c_{i+k-(k+1)} c_{i+k-2(k+1)} \cdots c_{i+k-q(k+1)}  a_{i-1 -q(k+1)+k} \cdots a_{i-1 -q(k+1)-k} \vspace{3 pt} \\
 &{}=  a_1^2 (a_2 \cdots a_n)^2  v c_{i+k-(k+1)} c_{i+k-2(k+1)} \cdots  c_{i+k-(q+1)(k+1)} a_{i+k-k-(q+1)(k+1)} \vspace{3 pt} \\
 &{}=  a_1^2 (a_2 \cdots a_n)^2  v c_{i-1} c_{i-1-(k+1)} \cdots c_{i-1-q(k+1)} a_{i-1-k-q(k+1)}.
\end{array}
\end{equation*}
Applying a sequence of transformations $r$, we get
$$ w_2 \succ \ldots \succ  a_1^2(a_2 \cdots a_n)^2 v  a_{i} b_{i} b_{i-(k+1)} \cdots b_{i-q(k+1)}  = w_2'.$$
By Lemma~\ref{joinn}, there exists $w_3\in \FM_n$ such that
$w_1,w_2'\geq w_3$.
\vspace{18 pt} \\
{\it Case 10:} Overlaps of the form $(\alpha, \zeta)$, [$(\delta,
\eta)$]. We apply $d_{i,v,q}$ to the subword of type $\zeta$ to get
$w_1$ and we apply $t_{i'}$ to the subword of type $\alpha$ to get
$w_2$. There are two kinds of overlap, the first kind is as follows:
\begin{equation*}
\begin{array}{r@{}l}
w&{}= a_1 \cdots a_n v  b_{i} b_{i-(k+1)} \cdots b_{i-q(k+1)} a_{i-1-q(k+1)}  \vspace{3 pt}   \\   
&{}=  \overbrace{
   a_1 \cdots a_n v' 
\makebox[0pt][l]{$\displaystyle{\underbrace{\phantom{
a_{i'}  \cdots  a_{i+1} \cdots a_{i'-1}
} }_{\text{$\alpha$, $t_{i'}$: $w_2$ } } } $}   
a_{i'}  \cdots  a_{i+1} \cdots a_{i'-1}      \cdots a_{i-2} b_{i-(k+1)} b_{i-2(k+1)} \cdots b_{i-q(k+1)} a_{i-1-q(k+1)}    
}^{\text{$\zeta$, $d_{i,v,q}$: $w_1$}} .   
\end{array}
\end{equation*}
\begin{equation*}
\begin{array}{r@{}l}
w_1&{}= a_1 \cdots a_n v'   a_{i'}  \cdots  a_{i}  a_{i+k} c_{i+k} c_{i+k-(k+1)} \cdots c_{i+k-q(k+1)}, \vspace{3 pt} \\
w_2&{}=  a_1 \cdots a_n v'   a_1  \cdots  a_na_{i'} \cdots a_{i-2}b_{i-(k+1)} b_{i-2(k+1)} \cdots b_{i-q(k+1)} a_{i-1-q(k+1)}.
\end{array}
\end{equation*}
Applying $h_i$ followed by a sequence of transformations $r$, we get
\begin{equation*}
\begin{array}{r@{}l}
w_1 &{}= a_1 \cdots a_n v'   a_{i'}  \cdots  a_{i}  a_{i+k} c_{i+k}
c_{i+k-(k+1)} \cdots c_{i+k-q(k+1)} \vspace{3 pt} \\
&{} \succ  a_1 \cdots a_n v'   a_{i'}  \cdots  a_{i-1}a_1 \cdots a_n  c_{i+k-(k+1)} c_{i+k-2(k+1)} \cdots c_{i+k-q(k+1)}  \vspace{3 pt} \\
&{} \succ   \ldots \succ a_1^2 (a_2 \cdots a_n)^2 v'   a_{i'}  \cdots  a_{i-1} c_{i+k-(k+1)} c_{i+k-2(k+1)} \cdots c_{i+k-q(k+1)} \vspace{3 pt}  \\
&{}= a_1^2 (a_2 \cdots a_n)^2 v'   a_{i'}  \cdots  a_{i-1} c_{i-1} c_{i-1-(k+1)} \cdots c_{i-1-(q-1)(k+1)}\vspace{3 pt} \\
&{}= w_1'.
\end{array}
\end{equation*}
Applying a sequence of transformations $r$, we get
\begin{equation*}
\begin{array}{r@{}l}
w_2&{}= a_1 \cdots a_n v'   a_1  \cdots  a_na_{i'} \cdots a_{i-2}b_{i-(k+1)} b_{i-2(k+1)} \cdots b_{i-q(k+1)} a_{i-1-q(k+1)} \vspace{3 pt} \\
&{}\succ a_1^2 (a_2 \cdots a_n)^2 v'   a_{i'} \cdots a_{i-2}b_{i-(k+1)} b_{i-2(k+1)} \cdots b_{i-q(k+1)} a_{i-1-q(k+1)} \vspace{3pt} \\
&{}= a_1^2 (a_2 \cdots a_n)^2 v'   a_{i'} \cdots a_{i-2}b_{i-k-1}
b_{i-k-1-(k+1)} \cdots \vspace{3 pt} \\
&{} \hspace{200 pt} \cdots b_{i-k-1-(q-1)(k+1)} a_{i-k-1-1-(q-1)(k+1)} \vspace{3 pt} \\
&{}=w_2'.
\end{array}
\end{equation*}
By Lemma~\ref{joinn}, there exists $w_3\in \FM_n$ such that
$w_1',w_2'\geq w_3$.
\vspace{6 pt} \\
The second kind of overlap is as follows:
\begin{equation*}
\hspace{-5 pt}
\begin{array}{r@{}l}
w &{}=  a_1 \cdots a_n v  b_{i} b_{i-(k+1)} \cdots b_{i-q(k+1)} a_{i-1-q(k+1)} \vspace{3 pt} \\  
&{}=  \overbrace{
   a_1 \cdots 
\makebox[0pt][l]{$\displaystyle{\underbrace{\phantom{
a_{i'} \cdots  a_n \cdots a_{i+1}  \cdots a_{i'-1}
} }_{\text{$\alpha$, $t_{i'}$: $w_2$ } } } $}   
 a_{i'} \cdots  a_n \cdots a_{i+1}  \cdots a_{i'-1} \cdots a_{i-2}b_{i-(k+1)} b_{i-2(k+1)} \cdots b_{i-q(k+1)} a_{i-1-q(k+1)}   
}^{\text{$\zeta$, $d_{i,v,q}$: $w_1$}} .      
\end{array}
\end{equation*}
Here $v=a_{n+1} \cdots a_{i}$. Thus
\begin{equation*}
\begin{array}{r@{}l}
 w_1&{}=a_1 \cdots a_{i'} \cdots  a_n \cdots  a_{i}  a_{i+k}  c_{i+k} c_{i+k-(k+1)} \cdots c_{i+k-q(k+1)}, \vspace{3 pt} \\
w_2&{}= a_1 \cdots a_{i'-1} a_1 a_2 \cdots a_n a_{i'} \cdots a_{i-2}b_{i-(k+1)} b_{i-2(k+1)} \cdots b_{i-q(k+1)} a_{i-1-q(k+1)}.
\end{array}
\end{equation*}
Applying $h_i$ followed by a sequence of transformations $r$, we get
\begin{equation*}
\begin{array}{r@{}l}
w_1&{}=a_1 \cdots a_{i'} \cdots  a_n \cdots  a_{i}  a_{i+k} c_{i+k}
c_{i+k-(k+1)} \cdots c_{i+k-q(k+1)} \vspace{3 pt}  \\
&{} \succ  a_1 \cdots a_{i'} \cdots  a_n \cdots  a_{i-1} a_1\cdots a_n c_{i+k-(k+1)} c_{i+k-2(k+1)} \cdots c_{i+k-q(k+1)} \vspace{3 pt}  \\
&{} \succ  \ldots \succ  a_1^2 (a_2 \cdots  a_n)^2 a_1 \cdots  a_{i-1}c_{i+k-(k+1)} c_{i+k-2(k+1)}\cdots c_{i+k-q(k+1)} \vspace{3 pt}  \\
&{} =  a_1^2 (a_2 \cdots  a_n)^2 a_1 \cdots  a_{i-1}c_{i-1} c_{i-1-(k+1)} \cdots c_{i-1-(q-1)(k+1)} \vspace{3 pt} \\
&{}=w_1'.
\end{array}
\end{equation*}
Applying a sequence of transformations $r$, we get
\begin{equation*}
\begin{array}{r@{}l}
w_2&{} = a_1 \cdots a_{i'-1} a_1 a_2 \cdots a_n a_{i'} \cdots a_{i-2}b_{i-(k+1)} b_{i-2(k+1)} \cdots b_{i-q(k+1)} a_{i-1-q(k+1)}  \vspace{3 pt} \\
&{} \succ   \ldots \succ  a_1^2 (a_2 \cdots a_n)^2 a_1 \cdots a_{i-2}b_{i-(k+1)} b_{i-2(k+1)} \cdots b_{i-q(k+1)} a_{i-1-q(k+1)} \vspace{3 pt} \\
&{} = a_1^2 (a_2 \cdots a_n)^2 a_1 \cdots a_{i-2} b_{i-k-1}
b_{i-k-1-(k+1)} \cdots \vspace{3 pt} \\
&{} \hspace{198 pt}  \cdots b_{i-k-1-(q-1)(k+1)} a_{i-k-1-1-(q-1)(k+1)} \vspace{3 pt} \\
&{}=w_2'.
\end{array}
\end{equation*}
\vspace{0 pt} \\
By Lemma~\ref{joinn}, there exists $w_3\in \FM_n$ such that
$w_1',w_2'\geq w_3$.
\vspace{18 pt} \\
{\it Case 11:} Overlaps of the form $(\alpha, \eta)$, [$(\delta,
\zeta)$]. We apply $e_{i,v,q}$ to the subword of type $\eta$ to get
$w_1$ and we apply $t_{i+k+1}$ to the subword of type $\alpha$ to
get $w_2$. There are two kinds of overlap. We consider the first
kind of overlap, the second kind, as in {\it case 10}, can be solved
similarly. We have
\begin{equation*}
\begin{array}{r@{}l}
w &{} = a_1 \cdots a_n v  c_{i} c_{i-(k+1)} \cdots c_{i-q(k+1)} a_{i-k-q(k+1)} \vspace{3pt} \\  
&{} =  \overbrace{
  a_1 \cdots a_n v' 
\makebox[0pt][l]{$\displaystyle{\underbrace{\phantom{
 a_{i+e+1} a_{i+e+2} \cdots a_{i+k-1}a_{i+k}
} }_{\text{$\alpha$, $t_{i+k+1}$: $w_2$ } } } $}   
a_{i+k+1} a_{i+k+2} \cdots a_{i+k-1}a_{i+k} a_{i+2k} \cdots   a_{i-3k}a_{i-2k} c_{i-(k+1)} c_{i-2(k+1)} \cdots c_{i-q(k+1)} a_{i-k-q(k+1)}    
}^{\text{$\eta$, $e_{i,v,q}$: $w_1$}} .  
\end{array}
\end{equation*}
\begin{equation*}
\hspace{-5 pt}
\begin{array}{r@{}l}
w_1&{}=a_1 \cdots a_n v' a_{i+k+1} a_{i+k+2} \cdots a_{i+k-1}  a_{i+1} b_{i+1} b_{i+1-(k+1)} \cdots b_{i+1-q(k+1)}, \vspace{3 pt} \\
w_2&{}= a_1 \cdots a_n v'   a_1  \cdots  a_n  a_{i+2k} \cdots a_{i-3k}a_{i-2k} c_{i-(k+1)} c_{i-2(k+1)} \cdots c_{i-q(k+1)} a_{i-k-q(k+1)}.
\end{array}
\end{equation*}
Applying $s_{i+k+1,v',1}$, we get
\begin{equation*}
\begin{array}{r@{}l}
w_1&{}=a_1 \cdots a_n v' a_{i+k+1} a_{i+k+2} \cdots a_{i+k-1}  a_{i+1} b_{i+1} b_{i+1-(k+1)} \cdots b_{i+1-q(k+1)} \vspace{3 pt}  \\
&{}= a_1 \cdots a_n v' a_{i+k+1} b_{i+k+1}   a_{i+1} b_{i+1} b_{i+1-(k+1)}b_{i+1-2(k+1)} \cdots b_{i+1-q(k+1)} \vspace{3 pt}  \\
&{}=  a_1 \cdots a_n v' a_{i+k+1} b_{i+k+1}  b_{i+k+1-(k+1)} a_{i-1} b_{i+1-(k+1)}b_{i+1-2(k+1)} \cdots b_{i+1-q(k+1)} \vspace{3 pt}  \\
&{} \succ  a_1^2 (a_2 \cdots a_n)^2 v' c_{i+k} b_{i+1-(k+1)}b_{i+1-2(k+1)} \cdots b_{i+1-q(k+1)}  \vspace{3 pt} \\
&{} =  a_1^2 (a_2 \cdots a_n)^2 v' a_{i+2k} \cdots a_{i-2k} a_{i-k} b_{i+1-(k+1)}b_{i+1-2(k+1)} \cdots b_{i+1-q(k+1)} \vspace{3 pt}  \\
&{} =  a_1^2 (a_2 \cdots a_n)^2 v' a_{i+2k} \cdots a_{i-2k} a_{i-k}
b_{i-k} b_{i-k-(k+1)}\cdots b_{i-k-(q-1)(k+1)} \vspace{3 pt} \\
&{}= w_1'.
\end{array}
\end{equation*}
Applying a sequence of transformations $r$, we get
\begin{equation*}
\hspace{-12 pt}
\begin{array}{r@{}l}
w_2 &{} = a_1 \cdots a_n v'   a_1  \cdots  a_n  a_{i+2k} \cdots   a_{i-3k}a_{i-2k} c_{i-(k+1)} c_{i-2(k+1)} \cdots c_{i-q(k+1)} a_{i-k-q(k+1)} \vspace{3 pt}\\
&{} \succ    \ldots \succ a_1^2 (a_2 \cdots a_n)^2 v'a_{i+2k} \cdots   a_{i-3k}a_{i-2k} c_{i-(k+1)} c_{i-2(k+1)} \cdots c_{i-q(k+1)} a_{i-k-q(k+1)}  \vspace{3 pt}\\
&{} =   a_1^2 (a_2 \cdots a_n)^2 v'a_{i+2k} \cdots   a_{i-3k}a_{i-2k}
c_{i-k-1} c_{i-k-1-(k+1)} \cdots \vspace{3 pt} \\
&{} \hspace{210 pt} \cdots c_{i-k-1-(q-1)(k+1)} a_{i-k-1-k-(q-1)(k+1)} \vspace{3 pt} \\
&{}=  w_2'.
\end{array}
\end{equation*} \vspace{-10 pt} \\
By Lemma~\ref{joinn}, there exists $w_3\in \FM_n$ such that
$w_1',w_2'\geq w_3$.
\vspace{18 pt} \\
{\it Case 12:} Overlaps of the form $(\alpha, \theta)$, [$(\delta,
\iota)$]. We apply $s_{i,v,q}$ to the subword of type $\theta$ to
get $w_1$ and we apply $t_{i'}$ to the subword of type $\alpha$ to
get $w_2$. There are two kinds of overlap. We consider the first
kind of overlap, the second kind, as in {\it case 10}, can be solved
similarly. We have
\begin{equation*}
\begin{array}{r@{}l}
w &{}=   a_1 \cdots a_n v  a_i  b_{i}b_{i-(k+1)} \cdots b_{i-q(k+1)} a_{i-1 -q(k+1)} \vspace{ 3 pt} \\  
&{}= \overbrace{
   a_1 \cdots a_n v' 
\makebox[0pt][l]{$\displaystyle{\underbrace{\phantom{
  a_{i'} \cdots a_i \cdots a_{i'-1} 
} }_{\text{$\alpha$, $t_{i'}$: $w_2$ } } } $}   
a_{i'} \cdots a_i \cdots a_{i'-1} \cdots   a_{i-2}b_{i-(k+1)} b_{i-2(k+1)} \cdots b_{i-q(k+1)} a_{i-1 -q(k+1)}    
}^{\text{$\theta$, $s_{i,v,q}$: $w_1$}}  .  
\end{array}
\end{equation*}
\begin{equation*}
\begin{array}{r@{}l}
w_1&{}=a_1^2(a_2 \cdots a_n)^2 v' a_{i'} \cdots a_{i-1} c_{i+k-(k+1)} c_{i+k-2(k+1)} \cdots c_{i+k-q(k+1)} \vspace{3 pt}, \\
w_2&{}= a_1 \cdots a_n v'   a_1  \cdots  a_n   a_{i'} \cdots   a_{i-2} b_{i-(k+1)}b_{i-2(k+1)} \cdots b_{i-q(k+1)}  a_{i-1 -q(k+1)}.
\end{array}
\end{equation*}
We have
\begin{equation*}
\begin{array}{r@{}l}
 w_1&{}=a_1^2(a_2 \cdots a_n)^2 v' a_{i'}
\cdots a_{i-1} c_{i+k-(k+1)} c_{i+k-2(k+1)} \cdots c_{i+k-q(k+1)} \vspace{3 pt} \\
&{}=a_1^2(a_2 \cdots a_n)^2 v' a_{i'} \cdots a_{i-1} c_{i-1}
c_{i-1-(k+1)} \cdots c_{i-1-(q-1)(k+1)}.
\end{array}
\end{equation*}
Applying a
sequence of transformations $r$, we get
\begin{equation*}
\begin{array}{r@{}l}
w_2&{}=a_1 \cdots a_n v'   a_1  \cdots  a_n   a_{i'} \cdots   a_{i-2} b_{i-(k+1)}b_{i-2(k+1)} \cdots b_{i-q(k+1)}  a_{i-1 -q(k+1)} \vspace{3 pt} \\
&{} \succ  a_1^2 (a_2 \cdots a_n)^2 v'  a_{i'} \cdots a_{i-2} b_{i-(k+1)}b_{i-2(k+1)} \cdots b_{i-q(k+1)}  a_{i-1 -q(k+1)} \vspace{3 pt} \\
&{} =  a_1^2 (a_2 \cdots a_n)^2 v'  a_{i'} \cdots a_{i-2}
b_{i-k-1}b_{i-k-1-(k+1)} \cdots \vspace{3 pt} \\
&{} \hspace{200 pt} \cdots b_{i-k-1-(q-1)(k+1)}  a_{i-k-1-1 -(q-1)(k+1)} \vspace{3 pt}\\
 &{}=w_2'.
\end{array}
\end{equation*}
By Lemma~\ref{joinn}, there exists $w_3\in \FM_n$ such that
$w_1,w_2'\geq w_3$.
\vspace{18 pt} \\
{\it Case 13:} Overlaps of the form $(\alpha, \iota)$, [$(\delta,
\theta)$]. We apply $u_{i,v,q}$ to the subword of type $\iota$ to
get $w_1$ and we apply $t_{i+1}$ to the subword of type $\alpha$ to
get $w_2$. There are two kinds of overlap. We consider the first
kind of overlap, the second kind, as in {\it case 10}, can be solved
similarly. We have
\begin{equation*}
\begin{array}{r@{}l}
w &{}=  a_1 \cdots a_n v a_i c_{i}c_{i-(k+1)} \cdots c_{i-q(k+1)} a_{i-k -q(k+1)} \vspace{3 pt} \\  
&{}=   \overbrace{
   a_1 \cdots a_n v' 
\makebox[0pt][l]{$\displaystyle{\underbrace{\phantom{
 a_{i+1}a_{i+2}\cdots a_{i-1} a_i 
} }_{\text{$\alpha$, $t_{i+1}$: $w_2$ } } } $}   
a_{i+1}a_{i+2}\cdots a_{i-1} a_i c_{i}c_{i-(k+1)} \cdots c_{i-q(k+1)} a_{i-k -q(k+1)}   
}^{\text{$\iota $, $u_{i,v,q}$: $w_1$}}.  
\end{array}
\end{equation*}
\begin{equation*}
\begin{array}{r@{}l}
w_1&{}=a_1^2(a_2 \cdots a_n)^2 v' a_{i+1}a_{i+2}\cdots a_{i-1} b_{i+1-(k+1)} b_{i+1-2(k+1)} \cdots b_{i+1-q(k+1)}, \vspace{3 pt} \\
w_2&{}= a_1 \cdots a_n v'   a_1  \cdots  a_n   c_{i}c_{i-(k+1)} \cdots c_{i-q(k+1)} a_{i-k -q(k+1)}.
\end{array}
\end{equation*}
We have
\begin{equation*}
\begin{array}{r@{}l}
w_1&{}=a_1^2(a_2 \cdots a_n)^2 v' a_{i+1}a_{i+2}\cdots a_{i-1} b_{i+1-(k+1)} b_{i+1-2(k+1)} \cdots b_{i+1-q(k+1)} \vspace{3 pt} \\
&{}= a_1^2(a_2 \cdots a_n)^2 v' a_{i+1}b_{i+1} b_{i+1-(k+1)}
b_{i+1-2(k+1)} \cdots b_{i+1-q(k+1)} .
\end{array}
\end{equation*}
Applying a sequence of transformations $r$, we get
\begin{equation*}
\begin{array}{r@{}l}
w_2&{}= a_1 \cdots a_n v'a_1  \cdots  a_n c_{i}c_{i-(k+1)} \cdots c_{i-q(k+1)} a_{i-k -q(k+1)} \vspace{3 pt} \\
&{}\succ  a_1^2 (a_2 \cdots  a_n)^2v' c_{i}c_{i-(k+1)} \cdots
c_{i-q(k+1)} a_{i-k -q(k+1)} \vspace{3 pt} \\
&{}= w_2'.
\end{array}
\end{equation*}
By Lemma~\ref{joinn}, there exists $w_3\in \FM_n$ such that
$w_1,w_2'\geq w_3$.
\vspace{18 pt} \\
{\it Case 14:} Overlaps of the form $(\beta, \zeta)$. We apply $d_{i,v,q}$ to the subword of type $\zeta$ to get
$w_1$ and we apply $r_{j,m}$ to the subword of type $\beta$ to get
$w_2$. There are two kinds of overlap. We consider the first kind of
overlap, the second kind, as in {\it case 10}, can be solved
similarly. We have
\begin{equation*}
\begin{array}{r@{}l}
 w &{}= a_1 \cdots a_n v b_{i} b_{i-(k+1)}\cdots b_{i-q(k+1)} a_{i-1-q(k+1)} \vspace{3 pt} \\  
&{}=  \overbrace{
   a_1 \cdots a_n v' 
\makebox[0pt][l]{$\displaystyle{\underbrace{\phantom{
a_j a_1^{m-1}a_1 \cdots a_{i+1} \cdots a_n 
} }_{\text{$\beta$, $r_{j,m}$: $w_2$ } } } $}   
a_j a_1^{m-1}a_1 \cdots a_{i+1} \cdots a_n  \cdots  a_{i-2}  b_{i-(k+1)} b_{i-2(k+1)}\cdots b_{i-q(k+1)} a_{i-1-q(k+1)}   
}^{\text{$\zeta$, $d_{i,v,q}$: $w_1$}}.  
\end{array}
\end{equation*}
\begin{equation*}
\begin{array}{r@{}l}
w_1&{}=a_1 \cdots a_n v' a_j a_1^{m-1}a_1 \cdots a_{i}  a_{i+k} c_{i+k} c_{i+k-(k+1)} \cdots c_{i+k-q(k+1)}, \vspace{3 pt} \\
w_2&{}= a_1 \cdots a_n v'a_1 \cdots a_n a_j a_1^{m-1} a_1  \cdots a_{i-2}   b_{i-(k+1)} b_{i-2(k+1)}\cdots b_{i-q(k+1)} a_{i-1-q(k+1)}.
\end{array}
\end{equation*}
Applying $h_i$ and a sequence of transformations $r$, we get
\begin{equation*}
\begin{array}{r@{}l}
w_1&{}=a_1 \cdots a_n v'a_j a_1^{m-1}a_1 \cdots a_{i}  a_{i+k}
c_{i+k} c_{i+k-(k+1)} \cdots c_{i+k-q(k+1)} \vspace{3 pt} \\
&{} \succ  a_1 \cdots a_n v'a_j a_1^{m-1}a_1 \cdots a_{i-1} a_1 a_2 \cdots a_n c_{i+k-(k+1)} c_{i+k-2(k+1)} \cdots c_{i+k-q(k+1)} \vspace{3 pt}  \\
&{}\succ  \ldots \succ  a_1^2 (a_2 \cdots a_n)^2 v'a_j a_1^{m-1}a_1 \cdots a_{i-1} c_{i+k-(k+1)} c_{i+k-2(k+1)} \cdots c_{i+k-q(k+1)} \vspace{3 pt}  \\
&{}= a_1^2 (a_2 \cdots a_n)^2 v'a_j a_1^{m-1}a_1 \cdots a_{i-1}
c_{i-1} c_{i-1-(k+1)} \cdots c_{i-1-(q-1)(k+1)} \vspace{3 pt} \\
&{}= w_1'.
\end{array}
\end{equation*}
Applying a sequence of transformations $r$, we get
\begin{equation*}
\begin{array}{r@{}l}
w_2&{} =  a_1 \cdots a_n v'a_1 \cdots a_n a_j a_1^{m-1} a_1  \cdots  a_{i-2}   b_{i-(k+1)} b_{i-2(k+1)}\cdots b_{i-q(k+1)} a_{i-1-q(k+1)} \vspace{3 pt} \\
&{} \succ  a_1^2 (a_2 \cdots a_n)^2 v' a_j a_1^{m-1} a_1  \cdots  a_{i-2}   b_{i-(k+1)} b_{i-2(k+1)}\cdots b_{i-q(k+1)} a_{i-1-q(k+1)} \vspace{3 pt} \\
&{} = a_1^2 (a_2 \cdots a_n)^2 v' a_j a_1^{m-1} a_1  \cdots  a_{i-2}
b_{i-k-1} b_{i-k-1-(k+1)} \cdots \vspace{3 pt}  \\
&{} \hspace{200 pt} \cdots b_{i-k-1-(q-1)(k+1)} a_{i-k-1-1-(q-1)(k+1)} \vspace{3 pt} \\
&{}= w_2'.
\end{array}
\end{equation*}
By Lemma~\ref{joinn}, there exists $w_3\in \FM_n$ such that
$w_1',w_2'\geq w_3$.
\vspace{18 pt} \\
{\it Case 15:} Overlaps of the form $(\beta, \eta)$. We apply $e_{n-k,v,q}$ to the subword of type $\eta$ to
get $w_1$ and we apply $r_{j,m}$ to the subword of type $\beta$ to
get $w_2$. There are two kinds of overlap. We consider the first
kind of overlap, the second kind, as in {\it case 10}, can be solved
similarly. We have
\begin{equation*}
\begin{array}{r@{}l}
w &{}= a_1 \cdots a_n v c_{n-k} c_{n-k-(k+1)} \cdots c_{n-k-q(k+1)} a_{n-k-k-q(k+1)} \vspace{3 pt} \\  
&{}=   \overbrace{
   a_1 \cdots a_n v' 
\makebox[0pt][l]{$\displaystyle{\underbrace{\phantom{
a_j a_1^{m-1} a_1 \cdots a_{n-1}  a_{n} 
} }_{\text{$\beta$, $r_{j,m}$: $w_2$ } } } $}   
a_j a_1^{m-1} a_1 \cdots a_{n-1}  a_{n} a_{n+k} \cdots a_{n-3k} c_{n-k-(k+1)} c_{n-k-2(k+1)} \cdots c_{n-k-q(k+1)} a_{n-k-k-q(k+1)}  
}^{\text{$\eta$, $e_{n-k,v,q}$: $w_1$}}.  
\end{array}
\end{equation*}
\begin{equation*}
\begin{array}{r@{}l}
w_1&{}=a_1 \cdots a_n v' a_j a_1^{m-1} a_1 \cdots a_{n-1}  a_{n-k+1} b_{n-k+1} b_{n-k+1-(k+1)} \cdots b_{n-k+1-q(k+1)}, \vspace{3 pt}\\
w_2&{}= a_1 \cdots a_n v' a_1 \cdots a_n a_j a_1^{m-1} a_{n+k} \cdots
a_{n-3k} c_{n-k-(k+1)} c_{n-k-2(k+1)} \cdots \vspace{3 pt} \\
&{} \hspace{245 pt} \cdots c_{n-k-q(k+1)}a_{n-k-k-q(k+1)} .
\end{array}
\end{equation*}
Applying $s_{1,v'a_j a_1^{m-1},1}$, we get
\begin{equation*}
\begin{array}{r@{}l}
w_1&{}=a_1 \cdots a_n v' a_j a_1^{m-1} a_1 \cdots a_{n-1}  a_{n-k+1} b_{n-k+1} b_{n-k+1-(k+1)} \cdots b_{n-k+1-q(k+1)} \vspace{4 pt} \\
&{}= a_1 \cdots a_n v' a_j a_1^{m-1} a_1 b_{1}b_{1-(k+1)}a_{n-k-1} b_{n-k+1-(k+1)} b_{n-k+1-2(k+1)} \cdots b_{n-k+1-q(k+1)}   \vspace{4 pt}  \\
&{} \succ  a_1^2 (a_1 \cdots a_n)^2 v' a_j a_1^{m-1} c_{1+k-(k+1)} b_{n-k+1-(k+1)} b_{n-k+1-2(k+1)} \cdots b_{n-k+1-q(k+1)}  \vspace{4 pt}  \\
&{} =  a_1^2 (a_1 \cdots a_n)^2 v' a_j a_1^{m-1} a_{k} \cdots a_{-2k}  b_{n-k+1-(k+1)} b_{n-k+1-2(k+1)} \cdots b_{n-k+1-q(k+1)}  \vspace{4 pt}  \\
&{} =  a_1^2 (a_1 \cdots a_n)^2 v' a_j a_1^{m-1} a_{k} \cdots a_{-2k}  b_{-2k} b_{-2k-(k+1)} \cdots b_{-2k-(q-1)(k+1)} \vspace{3 pt} \\
&{}=w_1'.
\end{array}
\end{equation*}
Applying a sequence of transformations $r$, we get
\begin{equation*}
\begin{array}{r@{}l}
w_2&{}=a_1 \cdots a_n v' a_1 \cdots a_n a_j a_1^{m-1} a_{n+k} \cdots a_{n-3k} c_{n-k-(k+1)} c_{n-k-2(k+1)} \cdots \vspace{3 pt} \\
&{} \hspace{250 pt} \cdots c_{n-k-q(k+1)} a_{n-k-k-q(k+1)}  \vspace{3 pt} \\
&{} \succ   \ldots   \succ a_1^2(a_2 \cdots a_n)^2 v' a_j a_1^{m-1} a_{n+k} \cdots a_{n-3k} c_{n-k-(k+1)} c_{n-k-2(k+1)} \cdots \vspace{3 pt} \\
&{} \hspace{250 pt} \cdots c_{n-k-q(k+1)} a_{n-k-k-q(k+1)} \vspace{3 pt} \\
&{}= a_1^2(a_2 \cdots a_n)^2 v' a_j a_1^{m-1} a_{n+k} \cdots a_{n-3k}
c_{n-2k-1} c_{n-2k-1-(k+1)} \cdots \vspace{3 pt} \\
&{} \hspace{185 pt} \cdots c_{n-2k-1-(q-1)(k+1)}
a_{n-2k-1-k-(q-1)(k+1)} \vspace{3 pt} \\
&{}=w_2'.
\end{array}
\end{equation*}
By Lemma~\ref{joinn}, there exists $w_3\in \FM_n$ such that
$w_1',w_2'\geq w_3$.
\vspace{18 pt} \\
{\it Case 16:} Overlaps of the form $(\beta, \theta)$. We apply $s_{i,v,q}$ to the subword of type $\theta$ to
get $w_1$ and we apply $r_{j,m}$ to the subword of type $\beta$ to
get $w_2$. There are two kinds of overlap. We consider the first
kind of overlap, the second kind, as in {\it case 10}, can be solved
similarly. We have
\begin{equation*}
\hspace{-5 pt}
\begin{array}{r@{}l}
w &{}= a_1 \cdots a_n v  a_i  b_{i} b_{i-(k+1)} \cdots b_{i-q(k+1)} a_{i-1 -q(k+1)} \vspace{3 pt} \\  
&{}=   \overbrace{
   a_1 \cdots a_n v' 
\makebox[0pt][l]{$\displaystyle{\underbrace{\phantom{
 a_ja_1^{m-1}a_1 \cdots a_{i } \cdots a_n  
} }_{\text{$\beta$, $r_{j,m}$: $w_2$ } } } $}   
 a_ja_1^{m-1}a_1 \cdots a_{i } \cdots a_n \cdots a_{i -2}  b_{i-(k+1)} b_{i-2(k+1)} \cdots b_{i-q(k+1)} a_{i-1 -q(k+1)}    
}^{\text{$\theta$, $s_{i,v,q}$: $w_1$}}.  
\end{array}
\end{equation*}
\begin{equation*}
\hspace{-5 pt}
\begin{array}{r@{}l}
w_1&{}=a_1^2 (a_2 \cdots a_n)^2 v'  a_ja_1^{m-1}a_1 \cdots a_{i-1} c_{i+k-(k+1)} c_{i+k-2(k+1)} \cdots c_{i+k-q(k+1)}, \vspace{3 pt} \\
w_2&{}= a_1 \cdots a_n v' a_1 \cdots a_n a_j a_1^{m-1}  a_{n+1} \cdots
a_{i -2}  b_{i-(k+1)} b_{i-2(k+1)} \cdots b_{i-q(k+1)} a_{i-1-q(k+1)} .
\end{array}
\end{equation*}
We have
\begin{equation*}
\begin{array}{r@{}l}
w_1&{}=a_1^2 (a_2 \cdots a_n)^2 v'  a_ja_1^{m-1}a_1 \cdots a_{i-1} c_{i+k-(k+1)} c_{i+k-2(k+1)} \cdots c_{i+k-q(k+1)} \vspace{3 pt}  \\
&{}=a_1^2 (a_2 \cdots a_n)^2 v'  a_ja_1^{m-1}a_1 \cdots a_{i-1}
c_{i-1} c_{i-1-(k+1)} \cdots c_{i-1-(q-1)(k+1)}.
\end{array}
\end{equation*}
Applying a sequence of transformations $r$, we get
\begin{equation*}
\hspace{-15 pt}
\begin{array}{r@{}l}
w_2&{}= a_1 \cdots a_n v' a_1 \cdots a_n a_j a_1^{m-1}  a_{n+1} \cdots a_{i -2}  b_{i-(k+1)} b_{i-2(k+1)} \cdots b_{i-q(k+1)} a_{i-1 -q(k+1)}. \vspace{3 pt} \\
&{}\succ \ldots \succ a_1^2 (a_2 \cdots a_n)^2 v'a_j a_1^{m-1}  a_{n+1} \cdots a_{i -2}  b_{i-(k+1)} b_{i-2(k+1)} \cdots b_{i-q(k+1)} a_{i-1 -q(k+1)} \vspace{3 pt} \\
&{} = a_1^2 (a_2 \cdots a_n)^2 v'a_j a_1^{m-1}  a_{n+1} \cdots a_{i
-2}  b_{i-k-1} b_{i-k-1-(k+1)} \cdots \vspace{3 pt} \\
&{} \hspace{ 213 pt} \cdots b_{i-k-1-(q-1)(k+1)}
a_{i-k-1-1 -(q-1)(k+1)} \vspace{3 pt}\\
&{}=w_2'.
\end{array}
\end{equation*}
By Lemma~\ref{joinn}, there exists $w_3\in \FM_n$ such that
$w_1,w_2'\geq w_3$.
\vspace{18 pt} \\
{\it Case 17:} Overlaps of the form $(\beta, \iota)$. We apply $u_{n,v,q}$ to the subword of type $\iota$ to get
$w_1$ and we apply $r_{j,m}$ to the subword of type $\beta$ to get
$w_2$. There are two kinds of overlap. We consider the first kind of
overlap, the second kind, as in {\it case 10}, can be solved
similarly. We have
\begin{equation*}
\begin{array}{r@{}l}
w &{}= a_1 \cdots a_n v  a_n  c_{n} c_{n-(k+1)} \cdots c_{n-q(k+1)}a_{n-k -q(k+1)} \vspace{3 pt} \\  
&{}=    \overbrace{
   a_1 \cdots a_n v' 
\makebox[0pt][l]{$\displaystyle{\underbrace{\phantom{
a_j a_1^{m-1} a_1 \cdots a_n  
} }_{\text{$\beta$, $r_{j,m}$: $w_2$ } } } $}   
a_j a_1^{m-1} a_1 \cdots a_n   c_{n} c_{n-(k+1)} \cdots c_{n-q(k+1)}a_{n-k -q(k+1)}   
}^{\text{$\iota$, $u_{n,v,q}$: $w_1$}}.  
\end{array}
\end{equation*}
\begin{equation*}
\begin{array}{r@{}l}
w_1&{}=a_1^2 (a_2 \cdots a_n)^2 v'  a_ja_1^{m-1}a_1 \cdots a_{n-1} b_{n+1-(k+1)} b_{n+1-2(k+1)} \cdots b_{n+1-q(k+1)}  \vspace{3 pt} \\
w_2&{}= a_1 \cdots a_n v' a_1 \cdots a_n a_j a_1^{m-1}  c_{n} c_{n-(k+1)} \cdots c_{n-q(k+1)}a_{n-k -q(k+1)}.
\end{array}
\end{equation*}
We have
\begin{equation*}
\begin{array}{r@{}l}
w_1&{}=a_1^2 (a_2 \cdots a_n)^2 v'
a_ja_1^{m-1}a_1 \cdots a_{n-1} b_{n+1-(k+1)} b_{n+1-2(k+1)} \cdots
b_{n+1-q(k+1)}  \vspace{3 pt} \\
&{}=a_1^2 (a_2 \cdots a_n)^2 v'  a_ja_1^{m-1}a_1 b_{n+1}
b_{n+1-(k+1)} b_{n+1-2(k+1)} \cdots b_{n+1-q(k+1)} .
\end{array}
\end{equation*}
Applying a sequence of transformations $r$ followed by
$e_{n,a_2\cdots a_n v'a_j a_1^{m-1} ,q}$, we get
\begin{equation*}
\begin{array}{r@{}l}
w_2&{}= a_1 \cdots a_n v' a_1 \cdots a_n a_j a_1^{m-1}  c_{n} c_{n-(k+1)} \cdots c_{n-q(k+1)}a_{n-k -q(k+1)}  \vspace{3 pt} \\
&{}\succ  \ldots \succ  a_1^2 (a_2 \cdots a_n)^2 v' a_j a_1^{m-1}  c_{n} c_{n-(k+1)} \cdots c_{n-q(k+1)}a_{n-k -q(k+1)} \vspace{3 pt} \\
&{}\succ  a_1^2 (a_2 \cdots a_n)^2 v'  a_ja_1^{m-1}a_1 b_{n+1}
b_{n+1-(k+1)} b_{n+1-2(k+1)} \cdots b_{n+1-q(k+1)} \vspace{3 pt} \\
&{}= w_1 .
\end{array}
\end{equation*}
\vspace{0 pt} \\
{\it Case 18:} Overlaps of the form $(\zeta, \zeta)$, [$(\eta,
\eta)$]. We apply $d_{i,v,q}$ to a subword of type $\zeta$ to get
$w_1$, and we apply $d_{i',v',q'}$ to the other subword of type
$\zeta$ to get $w_2$. There are two kinds of overlap. We consider
the first kind of overlap, the second kind, as in {\it case 10}, can
be solved similarly. We have
\begingroup
\fontsize{10.5 pt}{11pt}\selectfont
\begin{equation*}
\begin{array}{r@{}l}
w &{}= a_1 \cdots a_n v  b_{i} b_{i-(k+1)}  \cdots b_{i-q(k+1)}  a_{i-1-q(k+1)} \vspace{3 pt} \\  
&{}=   \overbrace{
\makebox[0pt][l]{$\displaystyle{\underbrace{\phantom{
a_1 \cdots a_n v' b_{i'}   \cdots b_{i'-(q'-1)(k+1)}  a_{i'+1-q'(k+1)} \cdots a_{i+1}  \cdots a_{i'-1-q'(k+1)}   
} }_{\text{$\zeta$, $d_{i',v',q'}$: $w_2$ } } } $}   
a_1 \cdots a_n v' b_{i'}   \cdots b_{i'-(q'-1)(k+1)}
a_{i'+1-q'(k+1)} \cdots a_{i+1}  \cdots a_{i'-1-q'(k+1)}   \cdots
a_{i-2}
 b_{i-(k+1)}  \cdots b_{i-q(k+1)}  a_{i-1-q(k+1)}   
}^{\text{$\zeta$, $d_{i,v,q}$: $w_1$}}  .  
\end{array}
\end{equation*}
\endgroup
Without loss of generality, it is assumed that the subwords $a_1 \cdots a_n$ at the beginning of
the two words of type $\zeta$ coincide. We have
\begin{equation*}
\hspace{-15 pt}
\begin{array}{r@{}l}
w_1 &{} = a_1 \cdots a_n v' b_{i'}   \cdots b_{i'-(q'-1)(k+1)}  a_{i'+1-q'(k+1)} \cdots a_{i} a_{i+k} c_{i+k} c_{i+k-(k+1)} \cdots c_{i+k-q(k+1)}, \vspace{3 pt}\\ 
w_2 &{} = a_1 \cdots a_n v' a_{i'+k} c_{i'+k} c_{i'+k-(k+1)} \cdots
c_{i'+k-q'(k+1)} \cdot    \vspace{3 pt}  \\
&{} \hspace{155 pt}  \cdot a_{i'-q'(k+1)}  \cdots a_{i-2} b_{i-(k+1)}
\cdots b_{i-q(k+1)}  a_{i-1-q(k+1)}.
\end{array}
\end{equation*}
Applying $h_i$ and a sequence of transformations $r$, we get
\begin{equation*}
\hspace{-12 pt}
\begin{array}{r@{}l}
w_1&{} =  a_1 \cdots a_n v' b_{i'}  \cdots b_{i'-(q'-1)(k+1)}  a_{i'+1-q'(k+1)} \cdots a_{i} a_{i+k} c_{i+k} c_{i+k-(k+1)} \cdots c_{i+k-q(k+1)}  \vspace{3 pt} \\
&{} =  a_1 \cdots a_n v' b_{i'}   \cdots b_{i'-(q'-1)(k+1)}  a_{i'+1-q'(k+1)} \cdots a_{i} a_{i+k} \cdots a_{i-k} \cdot \vspace{3 pt} \\
&{} \hspace{223 pt} \cdot  c_{i+k-(k+1)}c_{i+k-2(k+1)} \cdots c_{i+k-q(k+1)}  \vspace{3 pt} \\
&{} \succ  a_1 \cdots a_n v' b_{i'}   \cdots b_{i'-(q'-1)(k+1)}  a_{i'+1-q'(k+1)} \cdots a_{i-1} a_1 \cdots a_n \cdot \vspace{3 pt} \\
&{} \hspace{223 pt} \cdot c_{i+k-(k+1)}c_{i+k-2(k+1)} \cdots c_{i+k-q(k+1)}  \vspace{3 pt} \\
&{} \succ  \ldots \succ a_1^2 (a_2 \cdots a_n)^2 v' b_{i'}   \cdots b_{i'-(q'-1)(k+1)}  a_{i'+1-q'(k+1)} \cdots a_{i-1} \cdot \vspace{3 pt} \\
&{} \hspace{223 pt} \cdot c_{i+k-(k+1)}c_{i+k-2(k+1)} \cdots c_{i+k-q(k+1)}  \vspace{3 pt} \\
&{}=  a_1^2 (a_2 \cdots a_n)^2 v' b_{i'}  \cdots b_{i'-(q'-1)(k+1)} a_{i'+1-q'(k+1)} \cdots a_{i-1} \cdot \vspace{3 pt} \\
&{} \hspace{223 pt} \cdot c_{i-1}c_{i-1-(k+1)} \cdots c_{i-1-(q-1)(k+1)} \vspace{3 pt} \\
&{}= w_1'.
\end{array}
\end{equation*}
Applying $u_{i'+k,v',q'}$, we get
\begin{equation*}
\begin{array}{r@{}l}
w_2 &{}= a_1 \cdots a_n v' a_{i'+k} c_{i'+k} c_{i'+k-(k+1)} \cdots c_{i'+k-q'(k+1)} a_{i'-q'(k+1)}   \cdots   a_{i-2} \cdot \vspace{3 pt} \\
&{}\hspace{230 pt} \cdot  b_{i-(k+1)} \cdots b_{i-q(k+1)}  a_{i-1-q(k+1)} \vspace{3 pt}  \\
&{}= a_1 \cdots a_n v' a_{i'+k} c_{i'+k} c_{i'+k-(k+1)} \cdots c_{i'+k-q'(k+1)} a_{i'+k-k-q'(k+1)} \cdots   a_{i-2} \cdot \vspace{3 pt} \\
&{}\hspace{230 pt} \cdot b_{i-(k+1)}  \cdots b_{i-q(k+1)}  a_{i-1-q(k+1)} \vspace{3 pt} \\
&{}\succ  a_1^2 (a_2 \cdots a_n)^2 v' b_{i'+k+1-(k+1)}  \cdots b_{i'+k+1-q'(k+1)} a_{i'+1-q'(k+1)} \cdots   a_{i-2} \cdot \vspace{3 pt} \\
&{}\hspace{230 pt} \cdot b_{i-(k+1)}  \cdots b_{i-q(k+1)}  a_{i-1-q(k+1)} \vspace{3 pt}   \\
&{} =  a_1^2 (a_2 \cdots a_n)^2 v' b_{i'}  \cdots b_{i'-(q'-1)(k+1)} a_{i'+1-q'(k+1)} \cdots   a_{i-2} \cdot \vspace{3 pt} \\
&{}\hspace{230 pt} \cdot b_{i-(k+1)}  \cdots b_{i-q(k+1)}  a_{i-1-q(k+1)}     \vspace{3 pt}     \\
&{} =  a_1^2 (a_2 \cdots a_n)^2 v' b_{i'} \cdots b_{i'-(q'-1)(k+1)} a_{i'+1-q'(k+1)} \cdots  a_{i-2} \cdot \vspace{3 pt} \\
&{}\hspace{157 pt} \cdot b_{i-k-1} \cdots b_{i-k-1-(q-1)(k+1)} a_{i-k-1-1-(q-1)(k+1)} \vspace{3 pt} \\
&{}=  w_2'.
\end{array}
\end{equation*}
By Lemma~\ref{joinn}, there exists $w_3\in \FM_n$ such that
$w_1',w_2'\geq w_3$.
\vspace{18 pt} \\
{\it Case 19:} Overlaps of the form $(\zeta, \eta)$, [$(\eta,
\zeta)$]. We apply $e_{i,v,q}$ to the subword of type $\eta$ to get
$w_1$ and we apply $d_{ i+(q'+1)(k+1),v',q'}$ to the subword of type
$\zeta$ to get $w_2$. There are two kinds of overlap. We consider
the first kind of overlap, the second kind, as in {\it case 10}, can
be solved similarly. We have
\begin{equation*}
\begin{array}{r@{}l}
w &{}= a_1 \cdots a_n v c_{i} c_{i-(k+1)} \cdots c_{i-q(k+1)} a_{i-k-q(k+1)} \vspace{3 pt} \\  
&{}=  \overbrace{
\makebox[0pt][l]{$\displaystyle{\underbrace{\phantom{
a_1 \cdots a_n v' b_{i+(q'+1)(k+1)} \cdots b_{i+(q'+1)(k+1)-q'(k+1)} a_{i+k}  
} }_{\text{$\zeta$, $d_{ i+(q'+1)(k+1),v',q'}$: $w_2$ } } } $}   
a_1 \cdots a_n v' b_{i+(q'+1)(k+1)} \cdots b_{i+(q'+1)(k+1)-q'(k+1)} a_{i+k} \cdots a_{i-2k}  c_{i-(k+1)} \cdots c_{i-q(k+1)} a_{i-k-q(k+1)}   
}^{\text{$\eta$, $e_{i,v,q}$: $w_1$}}  .  
\end{array}
\end{equation*}
Without loss of generality, it is assumed that the subwords $a_1 \cdots a_n$ at the beginning of
the word of type $\eta$ and at the beginning of the word of type $\zeta$ coincide. We have
\begin{equation*}
\begin{array}{r@{}l}
w_1 &{}= a_1 \cdots a_n v' b_{i+(q'+1)(k+1)} \cdots
b_{i+(q'+1)(k+1)-q'(k+1)} a_{i+1} b_{i+1}  \cdots b_{i+1-q(k+1)}, \vspace{3 pt} \\
 w_2 &{}=  a_1 \cdots a_n v' a_{i+(q'+1)(k+1)+k}
c_{i+(q'+1)(k+1)+k}\cdots c_{i+(q'+1)(k+1)+k-q'(k+1)} \cdot \vspace{3 pt}  \\
&{} \hspace{155 pt} \cdot a_{i+2k} \cdots a_{i-2k} c_{i-(k+1)} \cdots c_{i-q(k+1)}
a_{i-k-q(k+1)}.
\end{array}
\end{equation*}
Applying $d_{
i+(q'+1)(k+1),v',q'+1}$, we get
\begin{equation*}
\begin{array}{r@{}l}
w_1 &{}= a_1 \cdots a_n v' b_{i+(q'+1)(k+1)} \cdots b_{i+(q'+1)(k+1)-q'(k+1)} a_{i+1}   b_{i+1}  \cdots b_{i+1-q(k+1)} \vspace{3 pt}  \\
&{}= a_1 \cdots a_n v' b_{i+(q'+1)(k+1)} \cdots b_{i+(q'+1)(k+1)-q'(k+1)} a_{i+1}a_{i+2}\cdots a_{i-1}  \cdot \vspace{3 pt} \\
&{}\hspace{260 pt} \cdot  b_{i+1-(k+1)}  \cdots b_{i+1-q(k+1)} \vspace{3 pt}  \\
&{}= a_1 \cdots a_n v' b_{i+(q'+1)(k+1)} \cdots b_{i+(q'+1)(k+1)-q'(k+1)} b_{i+(q'+1)(k+1)-(q'+1)(k+1)} a_{i-1}  \cdot \vspace{3 pt} \\
&{}\hspace{260 pt} \cdot  b_{i+1-(k+1)}  \cdots b_{i+1-q(k+1)} \vspace{3 pt}  \\
&{} \succ  a_1 \cdots a_n v' a_{i+(q'+1)(k+1)+k}  c_{i+(q'+1)(k+1)+k}\cdots c_{i+(q'+1)(k+1)+k-(q'+1)(k+1)}  \cdot \vspace{3 pt} \\
&{}\hspace{260 pt} \cdot b_{i+1-(k+1)}  \cdots b_{i+1-q(k+1)} \vspace{3 pt}  \\
&{}= a_1 \cdots a_n v' a_{i+(q'+1)(k+1)+k}  c_{i+(q'+1)(k+1)+k}\cdots c_{i+(q'+1)(k+1)+k-q'(k+1)} \cdot \vspace{3 pt}  \\
&{}\hspace{183 pt} \cdot   a_{i+2k} \cdots a_{i-2k} a_{i-k}    b_{i-k}\cdots b_{i-k-(q-1)(k+1)} \vspace{3 pt} \\
&{}=w_1'.
\end{array}
\end{equation*}
We have
\begin{equation*}
\begin{array}{r@{}l}
w_2&{} = a_1 \cdots a_n v' a_{i+(q'+1)(k+1)+k}  c_{i+(q'+1)(k+1)+k}\cdots c_{i+(q'+1)(k+1)+k-q'(k+1)}  \cdot \vspace{3 pt} \\
&{}\hspace{128 pt} \cdot a_{i+2k} \cdots a_{i-2k} c_{i-(k+1)} \cdots c_{i-q(k+1)} a_{i-k-q(k+1)} \vspace{3 pt}\\
 &{}= a_1 \cdots a_n v' a_{i+(q'+1)(k+1)+k}c_{i+(q'+1)(k+1)+k} \cdots c_{i+(q'+1)(k+1)+k-q'(k+1)}  \cdot \vspace{3 pt} \\
 &{}\hspace{128 pt} \cdot a_{i+2k} \cdots a_{i-2k}  c_{i-k-1} \cdots c_{i-k-1-(q-1)(k+1)} a_{i-k-q(k+1)}.
\end{array}
\end{equation*}
By Lemma~\ref{joinn}, there exists $w_3\in \FM_n$ such that
$w_1',w_2\geq w_3$.
\vspace{18 pt} \\
 {\it Case 20:} Overlaps of the form $(\zeta,
\theta)$, [$(\eta, \iota)$]. We apply $s_{i,v,q}$ to the subword of
type $\theta$ to get $w_1$ and we apply $d_{ i',v',q'}$ to the
subword of type $\zeta$ to get $w_2$. There are two kinds of
overlap. We consider the first kind of overlap, the second kind, as
in {\it case 10}, can be solved similarly. We have
\begingroup
\fontsize{11 pt}{11pt}\selectfont
\begin{equation*}
\begin{array}{r@{}l}
w &{}= a_1 \cdots a_n v   a_i  b_{i}\cdots b_{i-q(k+1)}  a_{i-1-q(k+1)} \vspace{3 pt} \\  
&{}=  \overbrace{
\makebox[0pt][l]{$\displaystyle{\underbrace{\phantom{
a_1 \cdots a_n v' b_{i'} \cdots b_{i'-(q'-1)(k+1)}  a_{i'+1 -q'(k+1)} \cdots a_i \cdots a_{i'-1 -q'(k+1)}  
} }_{\text{$\zeta$, $d_{ i',v',q'}$: $w_2$ } } } $}   
a_1 \cdots a_n v' b_{i'} \cdots b_{i'-(q'-1)(k+1)}  a_{i'+1 -q'(k+1)} \cdots a_i \cdots a_{i'-1 -q'(k+1)} \cdots a_{i -2} b_{i-(k+1)} \cdots b_{i-q(k+1)} a_{i-1 -q(k+1)}   
}^{\text{$\theta$, $s_{i,v,q}$: $w_1$}} .  
\end{array}
\end{equation*}
\endgroup
Without loss of generality, it is assumed that the subwords $a_1 \cdots a_n$ at the beginning of
the word of type $\theta$ and at the beginning of the word of type $\zeta$ coincide. We have
\begin{equation*}
\hspace{-9 pt}
\begin{array}{r@{}l}
w_1 &{} = a_1^2(a_2 \cdots a_n)^2 v' b_{i'} \cdots b_{i'-(q'-1)(k+1)}  a_{i'+1 -q'(k+1)} \cdots a_{i-1} c_{i+k-(k+1)}  \cdots c_{i+k-q(k+1)},   \vspace{3 pt}  \\
w_2&{} = a_1 \cdots a_n v' a_{i'+k} c_{i'+k} \cdots c_{i'+k-q'(k+1)} a_{i'-q'(k+1)} \cdots a_{i -2} \cdot \vspace{3 pt}  \\
&{} \hspace{230 pt} \cdot b_{i-(k+1)} \cdots b_{i-q(k+1)} a_{i-1 -q(k+1)}.
\end{array}
\end{equation*}
\vspace{-6 pt} \\
We have
\begin{equation*}
\hspace{-9 pt}
\begin{array}{r@{}l}
w_1 &{}= a_1^2(a_2 \cdots a_n)^2 v' b_{i'} \cdots b_{i'-(q'-1)(k+1)}  a_{i'+1 -q'(k+1)} \cdots a_{i-1} c_{i+k-(k+1)}  \cdots c_{i+k-q(k+1)}   \vspace{3 pt}  \\
&{}=a_1^2(a_2 \cdots a_n)^2 v' b_{i'} \cdots b_{i'-(q'-1)(k+1)}
a_{i'+1 -q'(k+1)} \cdots a_{i-1} c_{i-1}  \cdots c_{i-1-(q-1)(k+1)}.
\end{array}
\end{equation*}
Applying $u_{ i'+k,v',q'}$, we get
\begin{equation*}
\begin{array}{r@{}l}
w_2&{}= a_1 \cdots a_n v' a_{i'+k} c_{i'+k} \cdots c_{i'+k-q'(k+1)}  a_{i'-q'(k+1)} \cdots a_{i -2} \cdot \vspace{3 pt} \\
&{}\hspace{230 pt} \cdot  b_{i-(k+1)} \cdots b_{i-q(k+1)} a_{i-1 -q(k+1)} \vspace{3 pt} \\
&{}= a_1 \cdots a_n v' a_{i'+k} c_{i'+k} \cdots c_{i'+k-q'(k+1)}  a_{i'+k-k   -q'(k+1)} \cdots a_{i -2} \cdot  \vspace{3 pt} \\
&{}\hspace{230 pt} \cdot  b_{i-(k+1)} \cdots b_{i-q(k+1)} a_{i-1 -q(k+1)} \vspace{3 pt} \\
&{}\succ  a_1^2 (a_2 \cdots a_n)^2 v'  b_{i'+k+1-(k+1)} \cdots b_{i'+k+1-q'(k+1)}  a_{i'+1-q'(k+1)}  \cdot  \cdots a_{i -2}\cdot \vspace{3 pt} \\
&{}\hspace{230 pt} \cdot b_{i-(k+1)} \cdots b_{i-q(k+1)} a_{i-1 -q(k+1)} \vspace{3 pt} \\
&{}= a_1^2 (a_2 \cdots a_n)^2 v'  b_{i'} \cdots b_{i'-(q'-1)(k+1)}a_{i'+1-q'(k+1)} \cdots a_{i -2}  \cdot \vspace{3 pt} \\
&{}\hspace{200 pt} \cdot b_{i-k-1}\cdots b_{i-k-1-(q-1)(k+1)} a_{i-1 -q(k+1)} \vspace{3 pt} \\
&{}=w_2'.
\end{array}
\end{equation*}
By Lemma~\ref{joinn}, there exists $w_3\in \FM_n$ such that
$w_1,w_2'\geq w_3$.
\vspace{18 pt} \\
{\it Case 21:} Overlaps of the form $(\zeta, \iota)$, [$(\eta,
\theta)$]. We apply $u_{i,v,q}$ to the subword of type $\iota$ to
get $w_1$ and we apply $d_{ i+q'(k+1)+1,v',q'}$ to the subword of
type $\zeta$ to get $w_2$. There are two kinds of overlap. We
consider the first kind of overlap, the second kind, as in {\it case
10}, can be solved similarly. We have
\begin{equation*}
\hspace{-5 pt}
\begin{array}{r@{}l}
w &={} a_1 \cdots a_n v a_i c_{i} \cdots c_{i-q(k+1)} a_{i-k -q(k+1)} \vspace{3 pt}\\  
&={} \overbrace{
\makebox[0pt][l]{$\displaystyle{\underbrace{\phantom{
a_1 \cdots a_n v'  b_{i+q'(k+1)+1}   \cdots  b_{i+q'(k+1)+1-q'(k+1)} a_i  
} }_{\text{$\zeta$, $d_{ i+q'(k+1)+1,v',q'}$: $w_2$ } } } $}   
a_1 \cdots a_n v'  b_{i+q'(k+1)+1}   \cdots  b_{i+q'(k+1)+1-q'(k+1)} a_i   c_{i}c_{i-(k+1)} \cdots c_{i-q(k+1)} a_{i-k -q(k+1)}  
}^{\text{$\iota$, $u_{i,v,q}$: $w_1$}} .  
\end{array}
\end{equation*}
Without loss of generality, it is assumed that the subwords $a_1 \cdots a_n$ at the beginning of
the word of type $\iota$ and at the beginning of the word of type $\zeta$ coincide. We have
\begin{equation*}
\begin{array}{r@{}l}
 w_1 &{} = a_1^2 (a_2 \cdots a_n)^2 v'  b_{i+q'(k+1)+1}   \cdots  b_{i+q'(k+1)+1-q'(k+1)} b_{i+1-(k+1)} \cdots b_{i+1-q(k+1)}, \vspace{3 pt}  \\
 w_2 &{} = a_1 \cdots a_n v'  a_{i+q'(k+1)+1+k} c_{i+q'(k+1)+1+k}
 \cdots  c_{i+q'(k+1)+1+k-q'(k+1)}\cdot \vspace{3 pt} \\
&{} \hspace{255 pt} \cdot c_{i} \cdots c_{i-q(k+1)} a_{i-k -q(k+1)}.
\end{array}
\end{equation*}
We have
\begin{equation*}
\begin{array}{r@{}l}
w_1&{}= a_1^2 (a_2 \cdots a_n)^2 v'  b_{i+q'(k+1)+1}   \cdots  b_{i+q'(k+1)+1-q'(k+1)} b_{i+1-(k+1)} \cdots b_{i+1-q(k+1)} \vspace{3 pt} \\
&{}= a_1^2 (a_2 \cdots a_n)^2 v'  b_{i+(q'+1)(k+1)+1-(k+1)}
 \cdots  b_{i+(q'+1)(k+1)+1-(q'+1)(k+1)} \cdots \vspace{3 pt} \\
&{} \hspace{230 pt} \cdots b_{i+(q'+1)(k+1)+1-(q'+1+q)(k+1)}.
\end{array}
\end{equation*}
Applying $u_{ i+(q'+1)(k+1),v',q'+1+q}$, we get
\begin{equation*}
\begin{array}{r@{}l}
w_2 &{}= a_1 \cdots a_n v'a_{i+q'(k+1)+1+k} c_{i+q'(k+1)+1+k}   \cdots  c_{i+q'(k+1)+1+k-q'(k+1)}    c_{i}\cdot   \vspace{3 pt} \\
&{}\hspace{193 pt}\cdot c_{i-(k+1)} \cdots c_{i-q(k+1)} a_{i-k -q(k+1)}                  \vspace{3 pt}     \\
&{}= a_1 \cdots a_n v'a_{i+(q'+1)(k+1)} c_{i+(q'+1)(k+1)}   \cdots  c_{i+(q'+1)(k+1)-q'(k+1)}     \cdots  \vspace{3 pt}      \\
&{}\hspace{193 pt} \cdots     c_{i+(q'+1)(k+1)-(q'+1+q)(k+1)} a_{i-k -q(k+1)}               \vspace{3 pt}     \\
&{}\succ  a_1^2 (a_2 \cdots a_n)^2 v'  b_{i+(q'+1)(k+1)+1-(k+1)} \cdots b_{i+(q'+1)(k+1)+1-(q'+1)(k+1)} \cdots         \vspace{3 pt}     \\
&{} \hspace{193 pt} \cdots b_{i+(q'+1)(k+1)+1-(q'+1+q)(k+1)}    \vspace{3 pt}     \\
 &{}= w_1.
\end{array}
\end{equation*}
\vspace{0 pt} \\
{\it Case 22:} Overlaps of the form $(\theta, \zeta)$, [$(\iota,
\eta)$]. We apply $d_{i,v,q}$ to the subword of type $\zeta$ to get
$w_1$ and we apply $s_{i',v',q'}$ to the subword of type $\theta$ to
get $w_2$. There are two kinds of overlap. We consider the first
kind of overlap, the second kind, as in {\it case 10}, can be solved
similarly. Let $w_0=a_1\cdots a_n$.  We have
\begingroup
\fontsize{11 pt}{11pt}\selectfont
\begin{equation*}
\begin{array}{r@{}l}
w &{}=  a_1 \cdots a_n v  b_{i} \cdots b_{i-q(k+1)} a_{i-1 -q(k+1)} \vspace{3 pt} \\  
&{}= \overbrace{
\makebox[0pt][l]{$\displaystyle{\underbrace{\phantom{
w_0 v' a_{i'} b_{i'} \cdots b_{i'-(q'-1)(k+1)}a_{i'+1 -q'(k+1)} \cdots a_{i+1} \cdots a_{i'-1 -q'(k+1)} 
} }_{\text{$\theta$, $s_{i',v',q'}$: $w_2$ } } } $}   
w_0 v' a_{i'} b_{i'} \cdots b_{i'-(q'-1)(k+1)}a_{i'+1 -q'(k+1)} \cdots a_{i+1} \cdots a_{i'-1 -q'(k+1)}  \cdots  a_{i-2}b_{i-(k+1)}\cdots b_{i-q(k+1)} a_{i-1 -q(k+1)} 
}^{\text{$\zeta$, $d_{i,v,q}$: $w_1$}} .  
\end{array}
\end{equation*}
\endgroup
Without loss of generality, it is assumed that the subwords $w_0=a_1 \cdots a_n$ at the beginning of
the word of type $\zeta$ and at the beginning of the word of type $\theta$ coincide. We have
\begin{equation*}
\begin{array}{r@{}l}
 w_1 &{}= a_1 \cdots a_n v'  a_{i'}  b_{i'} \cdots b_{i'-(q'-1)(k+1)} a_{i'+1 -q'(k+1)} \cdots a_{i} a_{i+k} c_{i+k}  \cdots c_{i+k-q(k+1)}, \vspace{3 pt} \\
 w_2 &{}= a_1^2 (a_2 \cdots a_n)^2 v'  c_{i'+k-(k+1)} \cdots
c_{i'+k-q'(k+1)} a_{i'-q'(k+1)} \cdots  a_{i -2} \cdot \vspace{3 pt} \\
&{} \hspace{230 pt} \cdot b_{i-(k+1)} \cdots b_{i-q(k+1)}  a_{i-1 -q(k+1)}.
\end{array}
\end{equation*}
Applying $h_i$ and and a sequence of transformations $r$, we get
\begin{equation*}
\hspace{-8 pt}
\begin{array}{r@{}l}
w_1&{} =  a_1 \cdots a_n v'  a_{i'}  b_{i'} \cdots b_{i'-(q'-1)(k+1)} a_{i'+1 -q'(k+1)} \cdots a_{i} a_{i+k} c_{i+k}  \cdots c_{i+k-q(k+1)} \vspace{3 pt}  \\
&{}\succ  a_1 \cdots a_n v'  a_{i'}  b_{i'} \cdots b_{i'-(q'-1)(k+1)} a_{i'+1 -q'(k+1)} \cdots a_{i-1} a_1 \cdots a_n \cdot \vspace{3 pt}  \\
&{} \hspace{270 pt} \cdot c_{i+k-(k+1)} \cdots c_{i+k-q(k+1)} \vspace{3 pt}  \\
&{} \succ  \ldots \succ a_1^2(a_1  \cdots a_n)^2 v'  a_{i'}  b_{i'} \cdots b_{i'-(q'-1)(k+1)} a_{i'+1 -q'(k+1)} \cdots a_{i-1}  \cdot \vspace{3 pt}  \\
&{} \hspace{270 pt} \cdot c_{i+k-(k+1)} \cdots c_{i+k-q(k+1)} \vspace{3 pt} \\
&{} =  a_1^2(a_1  \cdots a_n)^2 v'  a_{i'}  b_{i'} \cdots
b_{i'-(q'-1)(k+1)} a_{i'+1 -q'(k+1)} \cdots a_{i-1}  c_{i-1} \cdots
c_{i-1-(q-1)(k+1)}  \vspace{3 pt} \\
&{}=w_1'.
\end{array}
\end{equation*}
We have
\begin{equation*}
\begin{array}{r@{}l}
w_2&{}= a_1^2 (a_2 \cdots a_n)^2 v'  c_{i'+k-(k+1)} \cdots c_{i'+k-q'(k+1)} a_{i'-q'(k+1)} \cdots  a_{i -2} \cdot \vspace{3 pt}  \\
&{} \hspace{232 pt} \cdot b_{i-(k+1)} \cdots b_{i-q(k+1)}  a_{i-1 -q(k+1)} \vspace{3 pt} \\
 &{}= a_1^2 (a_2 \cdots a_n)^2 v'  c_{i'-1}  \cdots c_{i'-1-(q'-1)(k+1)} a_{i'-q'(k+1)} \cdots  a_{i -2}  \cdot \vspace{3 pt}  \\
&{} \hspace{195 pt} \cdot b_{i-k-1} \cdots b_{i-k-1-(q-1)(k+1)}  a_{i-1 -q(k+1)}.
\end{array}
\end{equation*}
By Lemma~\ref{joinn}, there exists $w_3\in \FM_n$ such that
$w_1',w_2\geq w_3$.
\vspace{18 pt} \\
{\it Case 23:} Overlaps of the form $(\theta, \eta)$, [$(\iota,
\zeta)$]. We apply $e_{i,v,q}$ to the subword of type $\eta$ to get
$w_1$ and we apply $s_{ i+(q'+1)(k+1),v',q'}$ to the subword of type
$\theta$ to get $w_2$. There are two kinds of overlap. We consider
the first kind of overlap, the second kind, as in {\it case 10}, can
be solved similarly. Let $w_0=a_1\cdots a_n$. We have
\begingroup
\fontsize{11 pt}{11pt}\selectfont
\begin{equation*}
\begin{array}{r@{}l}
w&{}  =  w_0 v  c_{i}  \cdots c_{i-q(k+1)} a_{i-k-q(k+1)} \vspace{3 pt} \\  
&{}= \overbrace{
\makebox[0pt][l]{$\displaystyle{\underbrace{\phantom{
w_0 v' a_{i+k+1+q'(k+1)}  b_{i+k+1+q'(k+1)} \cdots b_{i+k+1+q'(k+1)-q'(k+1)} a_{i+k}  
} }_{\text{$\theta$, $s_{i+(q'+1)(k+1),v',q'}$: $w_2$ } } } $}   
w_0 v' a_{i+k+1+q'(k+1)}  b_{i+k+1+q'(k+1)} \cdots b_{i+k+1+q'(k+1)-q'(k+1)} a_{i+k} \cdots  a_{i-2k}  c_{i-(k+1)}  \cdots c_{i-q(k+1)} a_{i-k-q(k+1)}  
}^{\text{$\eta$, $e_{i,v,q}$: $w_1$}} .  
\end{array}
\end{equation*}
\endgroup
Without loss of generality, it is assumed that the subwords $w_0=a_1
\cdots a_n$ at the beginning of the word of type $\eta$ and at the
beginning of the word of type $\theta$ coincide. We have
\begin{equation*}
\hspace{-5 pt}
\begin{array}{r@{}l}
w_1&{}= a_1 \cdots a_n v' a_{i+k+1+q'(k+1)}  b_{i+k+1+q'(k+1)} \cdots b_{i+k+1+q'(k+1)-q'(k+1)}\cdot     \vspace{3 pt} \\
&{} \hspace{235 pt}  \cdot a_{i+1} b_{i+1}  \cdots b_{i+1-q(k+1)},   \vspace{3 pt} \\
w_2&{}=a_1^2(a_2 \cdots a_n)^2 v' c_{i+k+1+q'(k+1)+k-(k+1)} \cdots c_{i+k+1+q'(k+1)+k-q'(k+1)} a_{i+2k} \cdots  a_{i-2k} \cdot \vspace{3 pt}  \\
&{} \hspace{235 pt}  \cdot  c_{i-(k+1)}
 \cdots c_{i-q(k+1)} a_{i-k-q(k+1)} .
\end{array}
\end{equation*}
Applying $s_{(q'+1)(k+1),v',q'+1}$, we get
\begin{equation*}
\begin{array}{r@{}l}
w_1 &{} =  a_1 \cdots a_n v' a_{i+k+1+q'(k+1)}  b_{i+k+1+q'(k+1)} \cdots b_{i+k+1+q'(k+1)-q'(k+1)} \cdot \vspace{3 pt} \\
&{} \hspace{265 pt} \cdot   a_{i+1} b_{i+1}  \cdots b_{i+1-q(k+1)}  \vspace{3 pt} \\
&{} =  a_1 \cdots a_n v' a_{i+(q'+1)(k+1)}  b_{i+(q'+1)(k+1)} \cdots  b_{i+(q'+1)(k+1)-q'(k+1)} \cdot  \vspace{3 pt}  \\
&{} \hspace{125 pt} \cdot b_{i+(q'+1)(k+1)-(q'+1)(k+1)}  a_{i-1}b_{i+1-(k+1)}  \cdots b_{i+1-q(k+1)}  \vspace{3 pt} \\
&{}\succ  a_1^2 (a_2 \cdots a_n)^2 v'c_{i+(q'+1)(k+1)+k-(k+1)} \cdots c_{i+(q'+1)(k+1)+k-q'(k+1)}  \cdot  \vspace{3 pt}  \\
&{} \hspace{125 pt} \cdot c_{i+(q'+1)(k+1)+k-(q'+1)(k+1)} b_{i+1-(k+1)}  \cdots b_{i+1-q(k+1)}  \vspace{3 pt} \\
&{}= a_1^2 (a_2 \cdots a_n)^2 v'c_{i+(q'+1)(k+1)+k-(k+1)} \cdots c_{i+(q'+1)(k+1)+k-q'(k+1)}  a_{i+2k} \cdots a_{i-2k}  \cdot  \vspace{3 pt} \\
&{}  \hspace{240 pt} \cdot a_{i-k}b_{i+1-(k+1)}  \cdots b_{i+1-q(k+1)}  \vspace{3 pt} \\
&{}= a_1^2 (a_2 \cdots a_n)^2 v'c_{i+(q'+1)(k+1)+k-(k+1)} \cdots
c_{i+(q'+1)(k+1)+k-q'(k+1)}  a_{i+2k} \cdots a_{i-2k}\cdot  \vspace{3 pt} \\
&{}  \hspace{240 pt} \cdot  a_{i-k}  b_{i-k}  \cdots b_{i-k-(q-1)(k+1)}  \vspace{3 pt} \\
&{}=w_1'.
\end{array}
\end{equation*}
We have
\begin{equation*}
\begin{array}{r@{}l}
w_2&{}= a_1^2(a_2 \cdots a_n)^2 v' c_{i+(q'+1)(k+1)+k-(k+1)} \cdots c_{i+(q'+1)(k+1)+k-q'(k+1)} a_{i+2k} \cdots  a_{i-2k}\cdot \\
&{}\hspace{225 pt} \cdot  c_{i-(k+1)}  \cdots c_{i-q(k+1)} a_{i-k-q(k+1)} \\
&{}= a_1^2(a_2 \cdots a_n)^2 v' c_{i+(q'+1)(k+1)+k-(k+1)} \cdots c_{i+(q'+1)(k+1)+k-q'(k+1)} a_{i+2k} \cdots  a_{i-2k} \cdot \\
&{}\hspace{150 pt} \cdot  c_{i-k-1} \cdots c_{i-k-1-(q-1)(k+1)} a_{i-k-1-k-(q-1)(k+1)} .
\end{array}
\end{equation*}
By Lemma~\ref{joinn}, there exists $w_3\in \FM_n$ such that
$w_1',w_2\geq w_3$.
\vspace{18 pt} \\
{\it Case 24 :} Overlaps of the form $(\theta, \theta)$, [$(\iota,
\iota)$]. We apply $s_{i,v,q}$ to a subword of type $\theta$ to get
$w_1$, and we apply $s_{i',v',q'}$ to the other subword of type
$\theta$ to get $w_2$. There are two kinds of overlap. We consider
the first kind of overlap, the second kind, as in {\it case 10}, can
be solved similarly. Let $w_0= a_1 \cdots a_n$. We have
\begin{equation*}
\begin{array}{r@{}l}
w &{}=  a_1 \cdots a_n v a_i  b_{i}  \cdots b_{i-q(k+1)} a_{i-1 -q(k+1)} \vspace{3 pt} \\  
&{}= \overbrace{
\makebox[0pt][l]{$\displaystyle{\underbrace{\phantom{
w_0 v' a_{i'}   b_{i'}  \cdots b_{i'-(q'-1)(k+1)} a_{i'+1 -q'(k+1)} \cdots a_i \cdots a_{i'-1 -q'(k+1)} 
} }_{\text{$\theta$, $s_{i',v',q'}$: $w_2$ } } } $}   
w_0 v' a_{i'}   b_{i'}  \cdots b_{i'-(q'-1)(k+1)} a_{i'+1 -q'(k+1)} \cdots a_i \cdots a_{i'-1 -q'(k+1)} \cdots a_{i -2}  b_{i-(k+1)}  \cdots b_{i-q(k+1)} a_{i-1 -q(k+1)}  
}^{\text{$\theta$, $s_{i,v,q}$: $w_1$}}  .  
\end{array}
\end{equation*}
Without loss of generality, it is assumed that the subwords $w_0=  a_1 \cdots a_n$ at the beginning of
the two words of type $\theta$ coincide. We have
\begin{equation*}
\begin{array}{r@{}l}
w_1&{}= a_1^2(a_2 \cdots a_n)^2 v' a_{i'}   b_{i'}  \cdots b_{i'-(q'-1)(k+1)} a_{i'+1 -q'(k+1)} \cdots a_{i-1} \cdot \vspace{3 pt} \\
&{} \hspace{205 pt}\cdot c_{i+k-(k+1)}c_{i+k-2(k+1)} \cdots c_{i+k-q(k+1)},  \vspace{3 pt} \\
w_2&{}=a_1^2(a_2 \cdots a_n)^2  v' c_{i'+k-(k+1)} \cdots c_{i'+k-q'(k+1)}  a_{i'-q'(k+1)} \cdots  a_{i -2} \cdot \vspace{3 pt} \\
&{} \hspace{205 pt} \cdot b_{i-(k+1)} \cdots b_{i-q(k+1)} a_{i-1 -q(k+1)}.
\end{array}
\end{equation*}
We have
\begin{equation*}
\begin{array}{r@{}l}
w_1 &{}= a_1^2(a_2 \cdots a_n)^2 v' a_{i'}   b_{i'}  \cdots b_{i'-(q'-1)(k+1)} a_{i'+1 -q'(k+1)} \cdots a_{i-1} \cdot \vspace{3 pt} \\
&{} \hspace{260 pt} \cdot c_{i+k-(k+1)} \cdots c_{i+k-q(k+1)} \vspace{3 pt} \\
&{}= a_1^2(a_2 \cdots a_n)^2 v' a_{i'}   b_{i'}  \cdots b_{i'-(q'-1)(k+1)} a_{i'+1 -q'(k+1)} \cdots a_{i-1} \cdot \vspace{3 pt} \\
&{} \hspace{260 pt} \cdot c_{i-1} \cdots c_{i-1-(q-1)(k+1)} \vspace{3 pt},\\
w_2 &{}= a_1^2(a_2 \cdots a_n)^2 v' c_{i'+k-(k+1)}    \cdots c_{i'+k-q'(k+1)}  a_{i'-q'(k+1)} \cdots  a_{i -2}  \cdot \vspace{3 pt}\\
&{} \hspace{230 pt}  \cdot b_{i-(k+1)}  \cdots b_{i-q(k+1)} a_{i-1 -q(k+1)} \vspace{3 pt}\\
&{}= a_1^2(a_2 \cdots a_n)^2  v' c_{i'-1}   \cdots  c_{i'-1-(q'-1)(k+1)}  a_{i'-q'(k+1)} \cdots  a_{i -2} \cdot \vspace{3 pt}\\
&{} \hspace{155 pt} \cdot b_{i-k-1} \cdots b_{i-k-1-(q-1)(k+1)} a_{i-k-1-1 -(q-1)(k+1)}.
\end{array}
\end{equation*}
By Lemma~\ref{joinn}, there exists $w_3\in \FM_n$ such that
$w_1,w_2\geq w_3$.
\vspace{18 pt} \\
{\it Case 25:} Overlaps of the form $(\theta, \iota)$, [$(\iota,
\theta)$]. We apply $u_{i,v,q}$ to the subword of type $\iota$ to
get $w_1$, and we apply $s_{i+1+q'(k+1),v',q'}$ to the subword of
type $\theta$ to get $w_2$. There are two kinds of overlap. We
consider the first kind of overlap, the second kind, as in {\it case
10}, can be solved similarly. We have
\begin{equation*}
\hspace{-13 pt}
\begin{array}{r@{}l}
w &{}= a_1 \cdots a_n v a_i  c_{i}  \cdots c_{i-q(k+1)} a_{i-k -q(k+1)} \vspace{3 pt} \\  
&{}= \overbrace{
\makebox[0pt][l]{$\displaystyle{\underbrace{\phantom{
a_1 \cdots a_n v' a_{  i+1+q'(k+1)} b_{i+1+q'(k+1)}\cdots b_{i+1+q'(k+1)-q'(k+1)}a_i 
} }_{\text{$\theta$, $s_{i+1+q'(k+1),v',q'}$: $w_2$ } } } $}   
a_1 \cdots a_n v' a_{  i+1+q'(k+1)} b_{i+1+q'(k+1)}\cdots b_{i+1+q'(k+1)-q'(k+1)}a_i  c_{i}  \cdots c_{i-q(k+1)} a_{i-k -q(k+1)}  
}^{\text{$\iota$, $u_{i,v,q}$: $w_1$}}  .  
\end{array}
\end{equation*}
Without loss of generality, it is assumed that the subwords $a_1 \cdots a_n$ at the beginning of
the word of type $\iota$ and at the beginning of the word of type $\theta$ coincide. We have
\begin{equation*}
\begin{array}{r@{}l}
w_1 &{}=  a_1^2 (a_2 \cdots a_n)^2 v' a_{  i+1+q'(k+1)} b_{i+1+q'(k+1)}\cdots b_{i+1+q'(k+1)-q'(k+1)}\cdot  \vspace{3 pt}  \\
&{} \hspace{260 pt}    \cdot b_{i+1-(k+1)}  \cdots b_{i+1-q(k+1)}, \vspace{3 pt}  \\
 w_2 &{}= a_1^2(a_2 \cdots a_n)^2 v' c_{i+1+q'(k+1)+k-(k+1)}\cdots c_{i+1+q'(k+1)+k-q'(k+1)}\cdot  \vspace{3 pt}  \\
 &{} \hspace{260 pt}   \cdot c_{i} \cdots c_{i-q(k+1)} a_{i-k -q(k+1)}.
\end{array}
\end{equation*}
We have
\begin{equation*}
\begin{array}{r@{}l}
w_1 &{} =   a_1^2 (a_2 \cdots a_n)^2 v' a_{  i+1+q'(k+1)} b_{i+1+q'(k+1)}\cdots b_{i+1+q'(k+1)-q'(k+1)} \cdot  \vspace{3 pt} \\
&{} \hspace{260 pt}\cdot b_{i+1-(k+1)}  \cdots b_{i+1-q(k+1)}  \vspace{3 pt} \\
&{}=  a_1^2 (a_2 \cdots a_n)^2 v' a_{  i+1+q'(k+1)} b_{i+1+q'(k+1)}\cdots b_{i+1+q'(k+1)-q'(k+1)} \cdot  \vspace{3 pt}  \\
&{}\hspace{145 pt}\cdot b_{i+1+q'(k+1)-(q'+1)(k+1)} \cdots b_{i+1+q'(k+1)-(q'+q)(k+1)} \vspace{3 pt}, \\
w_2 &{}= a_1^2(a_2 \cdots a_n)^2 v' c_{i+1+q'(k+1)+k-(k+1)}\cdots c_{i+1+q'(k+1)+k-q'(k+1)}\cdot \vspace{3 pt}  \\
&{}\hspace{260 pt}\cdot  c_{i}  \cdots c_{i-q(k+1)} a_{i-k -q(k+1)}\vspace{3 pt}  \\
&{}= a_1^2(a_2 \cdots a_n)^2 v' c_{i+q'(k+1)}\cdots c_{i+q'(k+1)-(q'-1)(k+1)} c_{i}  \cdots c_{i-q(k+1)} a_{i-k -q(k+1)}  \vspace{3 pt} \\
&{}= a_1^2(a_2 \cdots a_n)^2 v' c_{i+q'(k+1)}\cdots c_{i+q'(k+1)-(q'-1)(k+1)}\cdot \vspace{3 pt}  \\
&{}\hspace{75 pt}\cdot c_{i+q'(k+1)-q'(k+1)}  \cdots c_{i+q'(k+1)-(q'+q)(k+1)} a_{i+q'(k+1)-k -(q'+q)(k+1)}.
\end{array}
\end{equation*}
By Lemma~\ref{joinn}, there exists $w_3\in \FM_n$ such that
$w_1,w_2\geq w_3$.
\vspace{18 pt}
\end{proof}

Note that, using the notation in the last part of
the proof of the above theorem, $S_n\setminus
a_1a_2\cdots a_nS_n=\{ a\in S_n\mid
\card \label{cardi} (\pi^{-1}(a))=1\}$. So, elements in this set have a
unique presentation in $S_{n}$.  In particular, we
have a natural isomorphism $S_{n}/(a_{1}\cdots
a_{n}S_{n})\cong \FM_{n}/\pi^{-1} (a_{1}\cdots
a_{n}S_{n})$ (again, we use the same notation for
the generators of $\FM_{n}$ and those of $S_{n}$).

\section{The algebra $K[S_n(H)]$ defined by dihedral type relations}\label{alg1}

Recall that if $R$ is a finitely generated algebra over a
field $K$, with a set of generators $a_1,\dots ,a_n$, and $\pi\colon
K\langle x_1,\dots ,x_n\rangle\longrightarrow R$ is the unique
homomorphism from the free algebra $K\langle x_1,\dots ,x_n\rangle$
to $R$ such that $\pi (x_i)=a_i$, for all $i$, and we have a fixed
order on the free monoid $\FM_n=\langle x_1,\dots ,x_n\rangle$
compatible with the multiplication in $\FM_n$, then $R$ is an
automaton algebra with respect to the presentation $\pi$ and the
fixed order on $\FM_n$ if and only if the ideal $I$ of $\FM_n$
consisting of all leading monomials of elements of $\ker(\pi)$ is a
regular language. A set of elements of $\FM_n$ is a regular language
if it is obtained from a finite subset of $\FM_n$ by applying a
finite number of operations of union, multiplication
and operation
$*$, defined by $T^*=\bigcup_{i\geq 1}T^{i}$, for $T\subseteq
\FM_n$.
Multiplication means multiplication of non-empty subsets of $\FM_n$: if $A$ and $B$ are non-empty  subsets of $\FM_n$, then
$AB=\{  ab \mid a\in A \mbox{ and } b\in B\}$. For example,
$\{ x_1,x_2\}x_3^*=\{ x_1x_3^{i}\mid i\geq 1\}\cup  \{ x_2x_3^{i}\mid i\geq 1\}$ and this is a regular language.
Note that $\FM_n$ is a regular language because
$\FM_n=\{ 1, x_1, x_2,\dots ,x_n\}^*$.
The set $N(R)=\FM_n\setminus I$ is called the set of normal
words of $R$, with respect to the given presentation of $R$ and the
fixed order on $\FM_n$. By \cite[Lemma of page~96]{ufnar}, $I$ is a
regular language if and only if $N(R)$ is a regular language.

Let $H$ be a subgroup of $\Sym_n$ that contains $\langle (1,2,\dots
,n)\rangle$ as a normal subgroup of index $2$, where $n>3$. We
denote $S_n(H)$ by $S_n$. An easy consequence of
Theorem~\ref{normalform} is the following result.

\begin{theorem}\label{autmaton} Let $K$ be a field. Then $K[S_n]$ is an automaton algebra with respect to the given presentation of $S_n$ and
the length-lexicographical order $\ll$  on $\FM_n$ generated by
$a_1\ll a_2\ll\dots\ll a_n$. (Here we use the same notation for the
generators of $\FM_n$ and those of $S_n$).
\end{theorem}

\begin{proof}
For $w\in\FM_n$, we define $w^*=\{ w^{i}\mid i\geq 0\}$. By
Theorem~\ref{normalform}, the set of normal words of the algebra
$K[S_n]$ is
\begin{equation*}
\begin{array}{r@{}l}
N(K[S_n])
&{}=a_1^*(\FM_n\setminus (I\cup a_2\cdots a_n\FM_n))\cup
(a_2\cdots
a_n)^*(\FM_n\setminus (I\cup a_1\FM_n)) \cup \vspace{3 pt} \\
&{}\hspace{16 pt} \cup\, a_1a_1^*(a_2\cdots a_n)(a_2\cdots a_n)^* \cdot \\
&{}\hspace{68 pt} \cdot (\FM_n\setminus
(I\cup a_1\FM_n\cup a_2\cdots a_n\FM_n\cup I_{\zeta } \cup I_{\eta }
\cup I_{\theta } \cup I_{\iota})),
\end{array}
\end{equation*}
where $I=\bigcup_{\sigma\in H}\FM_n a_{\sigma(1)}\cdots
a_{\sigma(n)}\FM_n$,
\begin{equation*}
\begin{array}{r@{}l}
I_{\zeta}&{}=\bigcup_{q=0}^{n-1}\bigcup_{i=n-k+1}^{n-1}\FM_n(b_{i}b_{i-(k+1)}\cdots b_{i-(n-1)(k+1)})^*  \cdot \vspace{3 pt} \\
&{} \hspace{197 pt} \cdot  b_{i}b_{i-(k+1)}\cdots b_{i-q(k+1)}a_{i-1-q(k+1)}\FM_n,\vspace{6 pt} \\
I_{\eta}&{}=\bigcup_{q=0}^{n-1}\bigcup_{i=0}^{n-k}\FM_n(c_{i}c_{i-(k+1)}\cdots c_{i-(n-1)(k+1)})^* \cdot \vspace{3 pt}\\
&{} \hspace{197 pt} \cdot c_{i}c_{i-(k+1)}\cdots c_{i-q(k+1)}a_{i-k-q(k+1)}\FM_n,  \vspace{6 pt} \\
I_{\theta}&{}=\bigcup_{q=0}^{n-1}\bigcup_{i=1}^{n}\FM_n a_i(b_{i}b_{i-(k+1)}\cdots b_{i-(n-1)(k+1)})^* \cdot \vspace{3 pt} \\
&{} \hspace{197 pt} \cdot b_{i}b_{i-(k+1)}\cdots b_{i-q(k+1)}a_{i-1-q(k+1)}\FM_n, \vspace{6 pt} \\
I_{\iota}&{}=\bigcup_{q=0}^{n-1}\bigcup_{i=1}^{n}\FM_n a_i(c_{i}c_{i-(k+1)}\cdots c_{i-(n-1)(k+1)})^* \cdot \vspace{3 pt} \\
&{} \hspace{197 pt} \cdot c_{i}c_{i-(k+1)}\cdots c_{i-q(k+1)}a_{i-k-q(k+1)}\FM_n.
\end{array}
\end{equation*}
By \cite[Lemma of page~96]{ufnar}, $N(K[S_n])$ is a regular language.
Therefore $K[S_n]$ is an automaton algebra.
\end{proof}

It follows from the beginning of Section \ref{monoidd}, that the
universal group of $S_n$, which means the group with the same presentation as $S_n$, is
\begin{equation*}
\begin{array}{r@{}l}
G_{n}=  \gr ( a_1 ,a_2 ,\ldots ,a_n  \mid    a_1 a_2 \cdots a_{n-1}a_n
 &{} = a_i a_{i+1} a_{i+2} \cdots a_{i-2} a_{i-1} \vspace{6 pt} \\
 &{} =a_i a_{i+k} a_{i+2 k} \cdots  a_{i-2 k} a_{i- k} , \ 1 \leq i \leq n ),
\end{array}
\end{equation*}
where $k^2 \equiv  1 \mod n$.
\vspace{0 pt} \\
\begin{theorem}
The universal group $G_n$ of $S_n(H)$ is a unique product group\index{unique product group}.
\end{theorem}

\begin{proof}
Consider the groups
$$ H_n= \gr ( a_2, \ldots, a_n \mid a_2 a_3 \cdots a_n = a_{1+k} a_{1+2k} \cdots a_{1+(n-1)k} ),$$
and $$ T_n = \gr (z) \times H_n .$$

Note that $(z,1)$ is a central element of $T_n$ and
\begin{eqnarray*}
(z, (a_2 a_3 \cdots a_n)^{-1})(1,a_2)(1,a_3) \cdots (1, a_n) &=& (z,1)   \\
\text{and} \hspace{12 pt} (z, (a_2 a_3 \cdots
a_n)^{-1})(1,a_{1+k})(1,a_{1+2k}) \cdots (1, a_{1+(n-1)k}) &=& (z,1)
\end{eqnarray*}
Hence there exists a unique homomorphism $f:G_n \to T_n $ such that
\begin{eqnarray*}
 f(a_i)&=&(1,a_i) \hspace{12 pt} \text{ for $i = 2, \ldots, n$, and} \\
 f(a_1)&=&(z, (a_2  \cdots a_n)^{-1}).
\end{eqnarray*}
Note that $a_1 a_2 \cdots a_n$ is a central element in $G_n$ (see page \pageref{centra}) and that
$a_2 a_3 \cdots a_n = a_{1+k} a_{1+2k} \cdots a_{1+(n-1)k} $ in
$G_n$. Furthermore, because the defining relations
of $G_n$ are homogeneous with respect to the degree of each $a_i$,
the subgroup of $G_n$ generated by $a_2, \ldots
, a_n$ intersects trivially with the subgroup generated by $a_1 a_2
\cdots a_n$. Hence there exists a unique homomorphism $g:T_n \to
G_n$ such that
\begin{eqnarray*}
 g(1, a_i)&=& a_i \hspace{12 pt} \text{for $i=2,\ldots, n$ and} \\
 g(z, 1) &=& a_1 a_2 \cdots a_n
\end{eqnarray*}
It is easy to check that $g=f^{-1}$. Hence $$G_n \cong T_n.$$

Recall that a group $G$ is said to be locally indicable\index{locally indicable} if every non-trivial
finitely generated subgroup of $G$ can be mapped homomorphically onto a finite cyclic group.
Such groups are right ordered \cite[Exercise 9, page 638]{passman} and  thus t.u.p.\index{t.u.p.} groups.
Indicable groups were first introduced by G. Higman\index{Higman, G.} in \cite{higman}.
It is well known that the group algebra of a u.p. group only has trivial units
(see \cite[Lemma 13.1.9]{passman})
and thus no zero divisors.

By a
result of Brodski\u{\i} \cite{brodskii}, every torsion free
one-relator group is locally indicable (see
\cite[Theorem~3.1]{howie} for a short proof).
Since the word $a_2a_3\cdots a_na_{1+(n-1)k}^{-1}a_{1+(n-2)k}^{-1}\cdots a_{1+k}^{-1}$
is not a proper power, by \cite[Theorem~4.12]{solitar}, $H_n$ is torsion free.
Hence $H_n$ is locally indicable, thus it is a unique product
group.
A direct product of u.p. groups is u.p,
thus $G_n$ is a u.p. group.
\end{proof}

Note that the monoid $S_n=S_n(H)$ is not cancellative. Thus $S_n$
is not embedded in $G_n$. Let $z=a_1 a_2 \ldots a_n \in S_n$, a central element in $S_n$.

\begin{lemma}\label{zS}
The subsemigroup $zS_n$ of $S_n$ is cancellative and $G_n = z S_n \langle z
\rangle ^{-1}$.
\end{lemma}
\begin{proof}
Since $z$ is central, it follows from the defining relations of $S_n$ that for every $w\in S_n$ there exists $u\in S_n$ such that $uw=z^l$ for some positive integer $l$.
Let $w_{k}=a_1^{i_{k}}(a_1a_2\cdots a_n)(a_2\cdots
a_n)^{j_{k}}b_{k}, k=1,2$, be two elements of $zS_{n}$ written in
the canonical form established in Theorem~\ref{normalform}. Suppose
that $ww_1=ww_2$, for some $w\in z S_n$. Then
$z^lw_1=z^lw_2$, for some positive integer $l$. But the canonical form of $z^lw_k$ is
$$z^lw_k=a_1^{i_{k}+l}(a_1a_2\cdots a_n)(a_2\cdots
a_n)^{j_{k}+l}b_{k},$$ for $k=1,2$. Hence $w_1=w_2$ and therefore,
$zS_n$ is left cancellative. Similarly one can prove that $zS_n$ is right cancellative. Therefore, it follows that $zS_{n}$ is cancellative and that the
central localization $zS_{n}\langle z\rangle ^{-1}$ is the group of
fractions of $zS_{n}$. Now it is easy to see that $G_n=zS_{n}\langle
z\rangle ^{-1}$.
\end{proof}
 In \cite[Section 2]{altalgebra} some general results on the structure of the algebra $K[S_n(H)]$ are proved (for arbitrary subgroups $H$
 of the symmetric group of degree $n$). In particular, if $z=a_1a_2\cdots a_n\in S_n(H)$ is central, then  $\mathcal{J}(K[S_n(H)]\subseteq K[zS_n(H)]$ \cite[Proposition 2.6]{altalgebra}.
 Using this in the present case, together with the fact that $zS_n$ is cancellative, we obtain the following result.
\begin{theorem}\label{Jrad}
Let $K$ be a field. Then $\mathcal{J}(K[S_n])= \{0 \}$.
\end{theorem}
\begin{proof}
We know that $G_n$ is a u.p. group. Hence $K[G_n]$ has only trivial
units (\cite{strojnowski} and
\cite[Lemma~13.1.9]{passman}).
Let
$\alpha \in \mathcal{J}(K[S_n])$.  By \cite[Proposition~2.6]{altalgebra},
$\alpha \in K[zS_n]$. There exists $\beta \in \mathcal{J}(K[S_n])$ such that
$$ \alpha + \beta + \alpha \beta = 0.$$
We also have that $\beta \in K[zS_n]$. Since
$$ (1+\alpha)(1+\beta)=1,$$
and $K[G_n]$ has only trivial units, there exist $\lambda_1,
\lambda_2 \in K$ and $w_1, w_2 \in zS_n\cup \{ 1\}$, such that
$$ 1 + \alpha = \lambda_1 w_1 \hspace{12 pt} \text{and} \hspace{12 pt} 1+\beta = \lambda_2 w_2$$
and $\lambda_1 \lambda_2 w_1 w_2 = 1$. Hence $w_1w_2=1$.
Because $S_n$ only has trivial units, we obtain that $w_1 =w_2 =1$
and
$$ \alpha = \lambda_1 w_1 - 1 = (\lambda_1 -1) \cdot 1.$$
Hence $\lambda_1 = 1$, and $\alpha = 0$.
\end{proof}

The proof of the following two results is as the proof of Lemma~2.5
and Proposition~2.6 of \cite{alghomshort}.

\begin{lemma}\label{prim}
$P=a_{1}a_{2}\cdots a_{n}S_{n}$ is a prime ideal of $S_n$. Moreover,
every prime ideal $Q$ of $K[S_{n}]$ such that $Q\cap S_{n}\neq
\emptyset $ contains $P$. Furthermore, $K[P]$ is a height one prime
ideal of $K[S_{n}]$.
\end{lemma}

Note that if $Q$ is a prime ideal of $S_{n}$, then $K[Q]$ is a prime
ideal of $K[S_{n}]$. Indeed, by the proof of Lemma~\ref{prim},
$P\subseteq Q$ and, hence, by
\cite[Proposition~24.2]{book-jan},
$K[S_n]/K[Q]$ is a prime monomial algebra.
In particular, since $n>2$, the submonoid $\langle a_1,a_2\rangle$ of $S_n/P$ is a free monoid of rank two. Therefore $K[S_n]/K[P]$ is not a PI-algebra.
Hence,  by a recent result of Okni\'nski
\cite[Theorem~0.1]{trichotomy},
and by a result of Bell and Colak
\cite[Theorem~1.2]{bell},
we obtain the following result.
\begin{proposition}\label{di-tri-chot}
Assume that $Q$ is a prime ideal of $S_n$. Then
$K[S_n]/K[Q]$ is either a prime PI-algebra\index{PI-algebra} or a (left and right)
primitive ring\index{primitive ring} or it has a nonzero locally nilpotent\index{locally nilpotent} Jacobson radical\index{Jacobson radical}.

Assume, furthermore, that $Q$ is finitely generated.
Then
$K[S_n]/K[Q]$ is either a prime PI-algebra or a (left and right)
primitive ring. In particular, $K[S_n]/K[P]$ is a (left and right)
primitive ring.
\end{proposition}

\vspace{30pt}
 \noindent \begin{tabular}{llllllll}
 F. Ced\'o && E. Jespers  \\
 Departament de Matem\`atiques &&  Department of Mathematics \\
 Universitat Aut\`onoma de Barcelona &&  Vrije Universiteit Brussel  \\
08193 Bellaterra (Barcelona), Spain    && Pleinlaan
2, 1050 Brussel, Belgium \\
 cedo@mat.uab.cat && efjesper@vub.ac.be\\
   &&   \\
G. Klein &&  \\ Department of Mathematics &&
\\  Vrije Universiteit Brussel && \\
Pleinlaan 2, 1050 Brussel, Belgium &&\\ gklein@vub.ac.be&&
\end{tabular}

\begin{thebibliography}{99}
\itemsep=-2pt
\bibitem{bell} J.P. Bell and P. Colak,  Primitivity of
finitely presented monomial algebras,  J. Pure Appl. Algebra 213
(2009), 1299--1305.
\bibitem{brodskii} S. D. Brodski\u{\i}  Equations over groups and groups with one defining
relation, Sibirsk. Mat. Zh. 25(2) (1984), 84--103; Siberian Math. J.
25(2) (1984), 235--251.
\bibitem{alt4algebra} F. Ced\'o, E. Jespers and J. Okni\'nski,
The radical of the four generated algebra of the
alternating type, Contemporary Math 499 (2009),
1--26.
\bibitem{alghomshort} F. Ced\'o, E. Jespers and J. Okni\'nski,
Finitely presented algebras and groups defined by
permutation relations, J. Pure App. Algebra 214
(2010), 1095--1102.
\bibitem{altalgebra} F. Ced\'o, E. Jespers and J. Okni\'nski,
Algebras and groups defined by permutation
relations of alternating type, J. Algebra 324 (2010) 1290--1313.
\bibitem{symcyclic} F. Ced\'o, E. Jespers and G. Klein,
Finitely presented monoids and algebras defined by permutation
relations of abelian type, J. Pure App. Algebra 216 (2012) 1033--1039.
\bibitem{quater} F. Ced\'o, E. Jespers and G. Klein,
Construction of a two unique product semigroup defined by permutation relations of quaternion type, in preparation.
\bibitem{cohn}
P.M. Cohn, Universal Algebra, revised edition, D. Reidel Publishing
Company, Dordrecht, Holland, 1981.
\bibitem{higman} G. Higman, The units of group-rings,
Proceedings of the London Mathematical Society, Second Series, 46 (1940), 231--248.
\bibitem{howie} J. Howie, A Short proof of a theorem of Brodskii,
Publicacions matematiques. Bellaterra : Servei de Publicacions de la
Universitat Aut\`onoma de Barcelona, V. 44 N. 2 (2000), p. 641--647.
\bibitem{book-jan} J. Okni\'nski,
Semigroup Algebras, Marcel Dekker, New York, 1991.
\bibitem{trichotomy} J. Okni\'nski,
Trichotomy for finitely generated monomial algebras, J. Algebra 417 (2014), 145--147.
\bibitem{passman} D.S. Passman, The Algebraic
Structure of Group Rings,   Wiley, New York,
1977.
\bibitem{solitar} W. Magnus, A. Karrass and D. Solitar,
Combinatorial group theory: Presentations of groups in terms of generators and relations,
Sec. Rev. Ed., Dover Publications, Inc., New York, 1976.
\bibitem{strojnowski} A. Strojnowski, A note on u.p. groups, Comm. Algebra, V. 8 N. 3 (1980), 231--234.


\bibitem{ufnar} V.A. Ufnarovski\u{\i}, Combinatorial and Asymptotic Methods
in Algebra, in: Encyclopedia of Mathematical Sciences vol. 57,
pp.1--196, Springer, 1995.

\end{thebibliography}
\end{document}